\documentclass[11pt]{article}
\usepackage{latexsym,a4wide}
\usepackage{amssymb}
\usepackage{amsmath}
\usepackage{color}
\usepackage{graphicx}
\usepackage{amsthm}
\usepackage{indentfirst}
\allowdisplaybreaks

\newcommand \bel {\begin{equation}\label}
\newcommand \ee {\end{equation}}
\newcommand \be {\begin{equation}}

\newcommand \RR {\mathbb R}
\newcommand \HH {\mathbb H}

\newcommand \CC {\mathbb C}
\newcommand \LL {\mathbb L}
\newcommand \NN {\mathbb N}

\newcommand \del \partial
\newcommand \Bcal {\mathcal B}

\newcommand \Dcal {\mathcal D}
\newcommand \Ccal {\mathcal C}

\newcommand \Lcal {\mathcal L}
\newcommand \Fcal {\mathcal F}

\newcommand \bei {\begin{itemize}}
\newcommand \eei {\end{itemize}}

\def \eps {\varepsilon}

\newtheorem{theorem}{\color{black}\indent Theorem}[section]
\newtheorem{lemma}{\color{black}\indent Lemma}[section]
\newtheorem{proposition}{\color{black}\indent Proposition}[section]

\newtheorem{remark}{\color{black}\indent Remark}[section]

\begin{document}
\large
\title{Dynamical behavior near explicit self-similar blow up solutions \\
for the Born-Infeld equation}
\author{
{\sc Weiping Yan}\thanks{School of Mathematics, Xiamen University, Xiamen 361000, P.R. China. Email: yanwp@xmu.edu.cn.}
}
\date{July 9, 2018}

\maketitle

\begin{abstract}  
This paper studies the dynamical behavior near a new family of explicit self-similar solutions for the one dimensional Born-Infeld equation. This quasilinear scalar field equation arises from nonlinear electromagnetism, as well as branes in string theory and minimal surfaces in Minkowski spacetimes. 
We show that both this model and the linear wave equation admit the same family of explicit timelike self-similar blow up solutions, meanwhile,
Lyapunov nonlinear stability of those self-similar blow up solutions are given inside a strictly proper subset of the backward light cone.

\end{abstract}

\tableofcontents


\section{Introduction and main results} 
\setcounter{equation}{0}

\subsection{Introduction}

In 1933-1934,
Born and Infeld \cite{Born0,Born} introduced a nonlinear electrodynamics theory to generalize
 the linear Maxwell theory. The related equation is called Born-Infeld equation. It also appears in geometric nonlinear theory of electromagnetism, which can be seen as the equation of graphs with zero mean curvature over a domain of the timelike $tx$-plane in Lorentz-Minkowski $\mathbb{L}^3(t,x,y)$. More precisely, the one dimensional Born-Infeld equation takes the following form
\bel{E1-1}
u_{tt}(1+u_x^2)-u_{xx}(1-u_t^2)=2u_tu_xu_{tx},\quad (t,x)\in\RR^+\times\RR,
\ee
where $u=u(t,x)$ is a real scalar-valued function. The solution of equation (\ref{E1-1}) also solves the timelike minimal surface equation
$$
\partial_t\left(\frac{\partial_t u}{\sqrt{1-(\partial_t u)^2+(\del_xu)^2}}\right)-\del_{x}\left(\frac{\del_{x}u}{\sqrt{1-(\partial_tu )^2+(\del_xu)^2}}\right)=0,
$$
which is the Euler-Lagrange equation of Lagrange action
$$
\mathcal{S}(u)=\int_{\mathbb{R}^+}\int_{\mathbb{R}}\sqrt{1-(\partial_t u)^2+(\del_x u)^2}dxdt.
$$

We supplement equation (\ref{E1-1}) with initial data
\bel{E1-1R0}
u(0,x)=u_0(x),\quad u_t(0,x)=u_1(x).
\ee

It is easy to check the one dimensional Born-Infeld equation (\ref{E1-1}) exhibits the scaling invariance 
\begin{equation}\label{E1-5}
u(t,x)\mapsto u_{\lambda}(t,x)=\lambda^{-1} u(\lambda t,\lambda x),\quad \lambda>0,
\end{equation}
and mass conservation 
$$
\int_{\mathbb{R}}\left(\frac{\partial_t u}{\sqrt{1-(\partial_t u)^2+(\del_xu)^2}}\right)dx~is~conserved~along~the~dynamics.
$$

The study of singularity is one of most important topics in physics and mathematics theory. It corresponds to a physical event such as the solution (e.g. a physical flow field) changing topology, or the emergence of a new structure (e.g. a tip, cusp). It can also imply that some essential physics is missing from the equation in question, which should thus be supplemented with additional terms.
Singularity behaviors have been observed in string theory \cite{W}.
Obviously, the Born-Infeld equation (\ref{E1-1}) is energy supercritical. Thus one expects smooth finite energy initial data to lead to finite time self-similar blow up solution.
Eggers and Hoppes \cite{Hop1} gave a detailed discussion on the existence of self-similar blow up solutions (not explicit self-similar solutions) for the Born-Infeld equation (\ref{E1-1}). 
They showed that it admits self-similar solutions
\begin{equation*}
u(t,x)=u_0-\hat{t}+\hat{t}^{a}h(\frac{x}{\hat{t}^b})+\ldots,
\end{equation*}
where $\hat{t}=t_0-t$ and $h(x)\varpropto A_{\pm}x^{\frac{2a}{a+1}}$ for $x\rightarrow\pm\infty$. In higher dimension case, they showed that the radially symmetric membranes equation admits self-similar solutions
\begin{equation*}
u(t,x)=-\hat{t}+\hat{t}^{a}h(\frac{x-x_0}{\hat{t}^b})+\ldots,
\end{equation*}
by analyzing the eikonal equation
\begin{equation*}
1-u_t^2+u_x^2=0.
\end{equation*}
Meanwhile, the swallowtail singularity was also been given by parametric string solution in \cite{Egg}. Alejo and Mu$\tilde{n}$oz \cite{Ale} obtained a sharp nonlinear scattering result for the Born-Infeld equation (\ref{E1-1}).
One can see \cite{Ba,Egg,Hop2,hop,K,Yang2,Lin,tian,Yang} for more discussion on the existence of solutions for the Born-Infeld equation and the membrane equation.

\subsection{Main results}

For the classicification of singularity in physics, there are timlike singularity, spacelike singularity and lightlike (null) singularity. 
To the Born-Infeld equation (\ref{E1-1}), if singular solution $u(t,x)$ of equation (\ref{E1-1}) satisfies $1-(\partial_t u)^2+(\del_xu)^2>0$, then it is called timelike singularity;
if singular solution $u(t,x)$ of equation (\ref{E1-1}) satisfies $1-(\partial_t u)^2+(\del_xu)^2<0$, then it is called spacelike singularity; if
singular solution $u(t,x)$ of equation (\ref{E1-1}) satisfies $1-(\partial_t u)^2+(\del_xu)^2=0$, then it is called lightlike singularity. Recently, the author of this paper \cite{Yan1} found a class of zero mean curvature equations (including the Born-Infeld equation (\ref{E1-1})) admits explicit finite time blow up solutions.
In the present paper, we show that both the Born-Infeld equation (\ref{E1-1}) and the linear wave equation
\bel{E1-2}
u_{tt}-u_{xx}=0,\quad \forall (t,x)\in\RR^+\times\RR
\ee
admit the same family of explicit self-similar solutions
$$
u_k(t,x)=k\ln(\frac{T-t+x}{T-t-x}),\quad |x|< T-t,\quad t\in[0,T),\quad \forall k\in\RR/\{0\},
$$
where constant $T$ denotes the maximal existence time. Meanwhile, those solutions of the Born-Infeld equation (\ref{E1-1})
are timelike self-similar solutions. After that, we prove those self-similar solutions are Lyapunov nonlinear stability inside a strictly proper subset of the backward light cone for the Born-Infeld equation (\ref{E1-1}). We now state main results of this paper.

\begin{theorem}
Both the Born-Infeld equation (\ref{E1-1}) and the linear wave equation (\ref{E1-2}) admit the same family of explicit self-similar solutions
\bel{E1-10}
u_k(t,x)=k\ln(\frac{T-t+x}{T-t-x}),\quad |x|< T-t,\quad t\in[0,T),\quad \forall k\in\RR/\{0\},
\ee
where $T$ denotes the maximal existence time.
Moreover, those explicit solutions of the Born-Infeld equation are timelike.

\end{theorem}

\begin{remark}
We remark that both the Born-Infeld equation (\ref{E1-1}) and the linear wave equation (\ref{E1-2}) also admit other kind of solutions except
blow up solutions (\ref{E1-10}). For example, the solution
$$
u(t,x)=g(t+x)-g(t-x),
$$
where $g(\cdot)$ is any smooth function.

\end{remark}

\begin{theorem}
The family of explicit timelike self-similar solutions (\ref{E1-10}) of the Born-Infeld equation (\ref{E1-1}) is Lyapunov nonlinear stability inside a strictly subset of the backward light cone $\Bcal_{t}$, that is,
for a sufficient small $\eps>0$ and a fixed constant $s\geq2$, if initial data (\ref{E1-1R0}) satisfies
$$
\|u_0(x)-u_k(0,x)\|_{\HH^{s}(\Bcal_0)}+\|u_1(x)-\del_tu_k(0,x)\|_{\HH^{s}(\Bcal_0)}<\eps,
$$
then equation (\ref{E1-1}) admits a local solution $u(t,x)$ such that 
$$
\sup_{t\in(0,\overline{T})}\| u(t,x) - u_k(t,x)\|_{\HH^s(\Bcal_t)}< \eps, 
$$
where $\HH^s(\Bcal_t)$ denotes the Sobolev space, and the domain
$$
\Bcal_{t}:=\Big\{(t,x)\Big| x\in[0,\delta(T-t)],\quad t\in[0,\overline{T}]\Big\},
$$
and constants $\delta\sim 1$ and $\overline{T}\in[T-\overline{\delta},T)$ with $0<\overline{\delta}\ll1$.

\end{theorem}

\begin{remark}
We should notice $T$ denotes the maximal existence time, which is a finite constant.
In the cone $|x|< T-t$,
$$
\frac{T-t+x}{T-t-x}>0,
$$
so solutions (\ref{E1-10}) are meaningful. 

Furthermore, by (\ref{E1-10}), we know that initial data is
$$
u_k(0,x)=k\ln(\frac{T+x}{T-x}),\quad |x|< T,
$$
which is a smooth function inside $(-T,T)$. 

The timelike boundary condition of the perturbation term is 
$$
w(t,x)|_{x\in\del\Bcal_t}:=\Big(u(t,x) - u_k(t,x)\Big)|_{x\in\del\Bcal_t}=0.
$$
\end{remark}


\subsection{Sketch of the proof}

The idea of finding explicit self-similar solutions is that we rewrite the Born-Infeld equation (\ref{E1-1}) in similarity coordinates, then solutions of ODE
\bel{YYY1-1}
(\rho^2-1)v_{\rho\rho}+2\rho v_{\rho}=0
\ee
give self-similar solutions of equation (\ref{E1-1}). Luckily, ODE (\ref{YYY1-1}) admits a family of explicit solutions. Meanwhile, we find that the linear wave equation (\ref{E1-2}) in similarity coordinates takes the form
$$
\phi_{\tau\tau}+\phi_{\tau}+2\rho \phi_{\rho}+2\rho \phi_{\tau\rho}+(\rho^2-1)\phi_{\rho\rho}=0,
$$
which also has the same steady equation (\ref{YYY1-1}). Thus explicit self-similar solutions $$u(t,x)=k\ln(\frac{T-t+x}{T-t-x}),\quad k\in \mathbb{R}\setminus \{0\}$$ of the Born-Infeld equation (\ref{E1-1}) are also 
explicit self-similar solutions of linear wave equation. This idea can be applied to find explicit self-similar solutions of other kinds of PDEs.

To show Lyapunov nonlinear stability of those self-similar solutions, we need to prove the local existence of regular solution for the perturbation equation, and this local solution is small in some Sobolev space. Since there is loss of derivatives, we have to employ Nash-Moser iteration scheme which has been used in \cite{Yan} 
to show local existence of regular small solution in the space $\Ccal^s_{2}:=\bigcap_{i= 0}^2\CC^i([0,T);\HH^{s-i}(\Omega))$ with a fixed constant $s\geq2$. Since coefficients of perturbation equation contain a term like
${1\over (T-t)^2-x^2}$, the blow up point $(T,0)$ can't be contained. Our result gives the analysis of dynamical behavior of solutions in a small $\delta$-neighborhood of blow up point $(T,0)$. Hence, the problem of nonlinear stability of self-similar solutions (\ref{E1-10}) at blow up point for the Born-Infeld equation (\ref{E1-1}) is still open.

\subsection{Notation}

Throughout this paper, the symbol $a\lesssim b$ means that there exists a positive constant $C$ such that $a\leq Cb$.
Let $\Omega$ be a bounded interval of $\RR$, we denote the usual norm of $\LL^2(\Omega)$ and $\HH^{s}(\Omega)$ by $\|\cdot\|_{\LL^2(\Omega)}$and $\|\cdot\|_{\HH^{s}(\Omega)}$ with constant $s\geq2$, respectively.
The space $\LL^2([0,T);\HH^{s}(\Omega))$ is equipped with the norm 
$$
\|u\|^2_{\LL^2([0,T);\HH^{s}(\Omega))}:=\int_0^{T}\|u(t,\cdot)\|^2_{\HH^{s}(\Omega)}dt.
$$
We also introduce the function space $\Ccal^s_{2}:=\bigcap_{i= 0}^2\CC^i([0,T);\HH^{s-i}(\Omega))$ with the norm
$$
\|u\|^2_{\Ccal^s_{2}}:=\sup_{t\in(0,T)}\sum_{i= 0}^2\|\partial^{i}_{t}u\|^2_{\HH^{s-i}(\Omega)}.
$$

The organization of this paper is as follows. In section 2, we prove both the Born-Infeld equation (\ref{E1-1}) and the linear wave admit the same family of explicit self-similar solutions. 
In the section 3, we show Lyapunov nonlinear stability of those self-simiilar solutions by using Nash-Moser iteration scheme.

\textbf{Acknowledgments.} 
The author expresses his sincerely thanks to the BICMR of Peking University and Professor Gang Tian for constant support and encouragement.
The author also expresses his sincere thanks to Prof. Baoping Liu for his many kind helps and suggestions, and thanks to Prof. Chao Xia for his discussion on elliptic equations.
The author is supported by NSFC No 11771359, and the Fundamental Research Funds for the Central Universities (Grant No. 20720190070, No.201709000061 and No. 20720180009).


\section{Self-similar solutions for the Born-Infeld equation and the linear wave equation}\setcounter{equation}{0}

In this section, we show both the Born-Infeld equation and the linear wave equation admit the same family of explicit self-similar solutions.
Since both the Born-Infeld equation and the linear wave equation will be rewritten in similarity coordinates, the maximal existence time $T$ should be a finite constant.
Otherwise, we can not use the similarity coordinates (see \cite{B1,BB}). 

We introduce the similarity coordinates
$$
\tau=-\log(T-t), \quad \rho=\frac{x}{T-t},
$$
then we denote by
$$
u(t,x)=\phi(-\log(T-t),\frac{x}{T-t}).
$$
Direct computation gives that
$$
u_t(t,x)=e^{\tau}(\phi_{\tau}+\rho \phi_{\rho}),
$$
$$
u_{tt}(t,x)=e^{2\tau}(\phi_{\tau\tau}+\phi_{\tau}+2\rho \phi_{\rho}+2\rho \phi_{\tau\rho}+\rho^2\phi_{\rho\rho}),
$$
$$
u_x(t,x)=e^{\tau}\phi_{\rho},
$$
$$
u_{xx}(t,x)=e^{2\tau}\phi_{\rho\rho},
$$
$$
u_{tx}(t,x)=e^{2\tau}(\phi_{\tau\rho}+\phi_{\rho}+\rho \phi_{\rho\rho}).
$$

Thus equation (\ref{E1-1}) is transformed into the one dimensional quasilinear wave equation
\bel{E2-1}
\aligned
v_{\tau\tau}-(1-\rho^2)v_{\rho\rho}&+v_{\tau}+2\rho v_{\rho}+2\rho v_{\tau\rho}+e^{2\tau}v_{\rho}^2(v_{\tau\tau}+v_{\tau}+2\rho v_{\rho}+2\rho v_{\tau\rho}+\rho^2v_{\rho\rho})\\
&+e^{2\tau}(v_{\tau}+\rho v_{\rho})^2v_{\rho\rho}-2e^{2\tau}v_{\rho}(v_{\tau}+\rho v_{\rho})(v_{\rho}+\rho v_{\rho\rho}+v_{\tau\rho})=0,
\endaligned
\ee
and the linear wave equation (\ref{E1-2}) is transformed into 
\bel{E2-1R1}
\phi_{\tau\tau}+\phi_{\tau}+2\rho \phi_{\rho}+2\rho \phi_{\tau\rho}+(\rho^2-1)\phi_{\rho\rho}=0.
\ee

Obviously, both equations (\ref{E2-1}) and (\ref{E2-1R1}) have the same steady equation 
$$
(\rho^2-1)v_{\rho\rho}+2\rho v_{\rho}=0,
$$
which is an ODE. Direct computation shows that it has a family of solutions
$$
v(\rho)=k\ln\frac{1+\rho}{1-\rho},
$$
where $k$ is an arbitrary constant in $\mathbb{R}\setminus \{0\}$.

Obviously, we should restrict $\rho$ to the domain
$$
\{\rho| \rho\in(-1,1)\}.
$$
Hence the Born-Infeld equation (\ref{E1-1}) has a family of explicit self-similar solutions
\bel{E2-2}
u_k(t,x)=k\ln(\frac{T-t+x}{T-t-x}),~~k\in \mathbb{R}\setminus \{0\},
\ee
where
\bel{E02-2}
|x|<T-t,\quad t\in[0,T).
\ee

Furthermore, it holds
$$
\partial_xu_k|_{x=0}=\frac{k}{T-t}\rightarrow+\infty,~~as~~t\rightarrow T^{-},
$$
and
$$
1+(u_k)_x^2-(u_k)_t^2=1+\frac{4}{(T-t)^2-x^2}>0,
$$
Hence, this kind of explicit self-similar solutions $u_k(t,x)$ given in (\ref{E2-2}) is timelike singularity.


\section{Lyapunov nonlinear stability of self-similar solutions}
Let $\delta\sim1$ be a positive constant.
Since explicit self-similar solutions (\ref{E2-2}) are defined in domain (\ref{E02-2}), 
we will consider the dynamical behavior of self-similar solutions inside the domain
$$
\Bcal_t:=\Big\{(t,x) \Big| x\in[0,\delta(T-t)],\quad t\in[0,\overline{T}]\Big\},
$$
where $\overline{T}\in[T-\overline{\delta},T)$ with $0<\overline{\delta}\ll1$.

Let 
\bel{E3-0}
u(t,x)=u_k(t,x)+w(t,x),
\ee
where $u_k(t,x)$ is given in (\ref{E2-2}).

Note that $k\in\RR/\{0\}$ is a constant. We only deal with the case of $k=1$ in (\ref{E3-0}) for convenience. 
Substituting (\ref{E3-0}) into (\ref{E1-1}), the perturbation equation takes the form
\bel{E3-1}
\aligned
w_{tt}&-\frac{((T-t)^2-x^2)^2-4x^2}{((T-t)^2-x^2)^2+4(T-t)^2}w_{xx}
-\frac{8(T-t)}{((T-t)^2-x^2)^2+4(T-t)^2}w_t\\
&+\frac{8x}{((T-t)^2-x^2)^2+4(T-t)^2}w_x-\frac{8x(T-t)}{((T-t)^2-x^2)^2+4(T-t)^2}w_{tx}\\
&+\frac{((T-t)^2-x^2)^2}{((T-t)^2-x^2)^2+4(T-t)^2}\Big[(\frac{4x(T-t)}{((T-t)^2-x^2)^2}+w_{tt})w_x^2+(\frac{4x(T-t)}{((T-t)^2-x^2)^2}+w_{xx})w_t^2\\
&-2(\frac{2(T-t)}{(T-t)^2-x^2}w_t+\frac{2x}{(T-t)^2-x^2}w_{x})w_{tx}
+2(\frac{2(T-t)}{(T-t)^2-x^2}w_{tt}-\frac{2((T-t)^2+x^2)}{((T-t)^2-x^2)^2}w_{t})w_{x}\\
&+2(\frac{2x}{(T-t)^2-x^2}w_{xx}-w_xw_{tt})w_t\Big]=0,\quad (t,x)\in\Bcal_t,
\endaligned
\ee
with small initial data
\bel{E03-01}
w(0,x)=\eps w_0(x),\quad w_t(0,x)=\eps w_1(x).
\ee

Obviously, it holds
$$
\aligned
&\eps w_0(x)=u_0(x)-\ln(\frac{T+x}{T-x}),\\
&\eps w_1(x)=u_1(x)-\frac{2x}{T^2-x^2},
\endaligned
$$
where $\eps\ll1$ is a positive small constant.

We supplement the timelike boundary condition
\bel{E03-02}
w(t,x)|_{x\in\del\Omega}=0,
\ee
where $\del\Omega=\{x=0\}\cup\{x=\delta(T-t)\}$.

In the domain $\Bcal_t$, one can see
$$
\frac{((T-t)^2-x^2)^2-4x^2}{((T-t)^2-x^2)^2+4(T-t)^2}=\frac{(1-\frac{x^2}{(T-t)^2})^2-\frac{4x^2}{(T-t)^4}}{(1-\frac{x^2}{(T-t)^2})^2+\frac{4}{(T-t)^2}}
\in[-\frac{\frac{4\delta^2}{(T-t)^2}-(1-\delta^2)^2}{(1-\delta^2)^2+\frac{4}{(T-t)^2}},\frac{1}{1+\frac{4}{(T-t)^2}}].
$$

Note that $\delta\sim 1$. It holds
$$
-\frac{\frac{4\delta^2}{(T-t)^2}-(1-\delta^2)^2}{(1-\delta^2)^2+\frac{4}{(T-t)^2}}<0.
$$
Thus it implies that quasilinear equation (\ref{E3-1}) is an elliptic-hyperbolic mixed-type equation, and (\ref{E3-1}) is degenerate on line 
$$
x=-1+\sqrt{1+(T-t)^2}.
$$
More precisely, equation (\ref{E3-1}) is a strictly hyperbolic equation in the hyperbolic domain
$$
\Omega_1:=[0,-1+\sqrt{1+(T-t)^2}),
$$
and it is a degenerate elliptic equation in the elliptic domain
$$
\Dcal:=\{(t,x)| (t,x)\in[0,\overline{T}]\times\Omega_2\},
$$
where
$$
\Omega_2:=[-1+\sqrt{1+(T-t)^2},\delta(T-t)].
$$

As we known, a classical mixed type equation is Tricomi-type equation. Morawetz \cite{Mora} gave a detailed account of the historical background and known results on mixed type equations and transonic flows.
In our case, there is a degenerate line, thus the whole domain is divided into elliptic domain and hyperbolic domain. Moreover, Since we require the solution vanishes along the degenerate line,
one can prescribe the solution arbitrarily (as a smooth function) on the degenerate line. This idea has been used to solve local existence of smooth solutions for a class of mixed-type Monge-Amp\'{e}re equations \cite{Han,Han1}.


\subsection{The linearized equation with singular coefficients}
We consider the linearized equation of (\ref{E3-1}) around a fixed function $\overline{w}$ with an external force $f(t,x)$ as follows
\bel{E3-2}
a(t,x)h_{tt}-b(t,x)h_{xx}-c(t,x)h_t+d(t,x)h_x-e(t,x)h_{tx}=f(t,x),
\ee
where
\bel{E3-2x}
\aligned
a(t,x)&:=1+j(t,x)\Big(\overline{w}_x^2+\frac{4\overline{w}_x(T-t)}{(T-t)^2-x^2}\Big),\\
b(t,x)&:=\frac{\Big((T-t)^2-x^2\Big)^2-4x^2}{\Big((T-t)^2-x^2\Big)^2+4(T-t)^2}-j(t,x)\Big(\overline{w}_t^2+\frac{4\overline{w}_t(T-t)}{(T-t)^2-x^2}\Big),\\
c(t,x)&:=\frac{8(T-t)}{\Big((T-t)^2-x^2\Big)^2+4(T-t)^2}-j(t,x)\Big[2\overline{w}_t\Big(\frac{4x(T-t)}{((T-t)^2-x^2)^2}+\overline{w}_{xx}\Big)\\
&+\frac{4\overline{w}_{tx}(T-t)}{(T-t)^2-x^2}+\frac{2\overline{w}_x((T-t)^2+x^2)}{((T-t)^2-x^2)^2}-2(\frac{2x\overline{w}_{xx}}{(T-t)^2-x^2}-\overline{w}_x\overline{w}_{tx})\Big],\\
d(t,x)&:=\frac{8x}{\Big((T-t)^2-x^2\Big)^2+4(T-t)^2}+j(t,x)\Big[\frac{8x(T-t)\overline{w}_x}{((T-t)^2-x^2)^2}\\
&+2\overline{w}_{tt}\Big(\overline{w}_x+\frac{2(T-t)}{(T-t)^2-x^2}\Big)-\frac{4x\overline{w}_{tx}}{(T-t)^2-x^2}-\frac{4\overline{w}_t((T-t)^2+x^2)}{((T-t)^2-x^2)^2}-2\overline{w}_{tx}w_t\Big],\\
e(t,x)&:=\frac{8x(T-t)}{((T-t)^2-x^2)^2+4(T-t)^2}+2j(t,x)\Big[\frac{2(T-t)}{(T-t)^2-x^2}\overline{w}_t+\frac{2x\overline{w}_x}{(T-t)^2-x^2}+2\overline{w}_x\overline{w}_t\Big],
\endaligned
\ee
and
$$
j(t,x):=\frac{((T-t)^2-x^2)^2}{((T-t)^2-x^2)^2+4(T-t)^2}.
$$

Since the dynamical behavior of self-similar solutions (\ref{E2-2}) is considered inside the lightcone, the singularity of coefficients may only take place at $\frac{x}{T-t}=0$ or $x=-1+\sqrt{1+(T-t)^2}$ or $\frac{x}{T-t}=\delta$. Moreover, it is easy to see that
$$
\aligned
j(t,x)=\left\{
\begin{array}{lll}
&&\frac{1}{4+(T-t)^2},\quad when \quad\frac{x}{T-t}=0,\\
&&\frac{2+(T-t)^2-2\sqrt{1+(T-t)^2}}{2+2(T-t)^2-2\sqrt{1+(T-t)^2}}\quad when \quad x=-1+\sqrt{1+(T-t)^2},\\
&&\frac{(1-\delta^2)^2(T-t)^2}{4+(1-\delta^2)^2(T-t)^2},\quad\quad when \quad \frac{x}{T-t}=\delta,
\end{array}
\right.
\endaligned
$$
which means that there is no singular point of $j(t,x)$. In what follows, we denote $j(t,x)$ by itself at $\frac{x}{T-t}=0$, $x=-1+\sqrt{1+(T-t)^2}$
and $\frac{x}{T-t}=\delta$ for convenience.
So asymptotic behavior of coefficients $a(t,x)$, $b(t,x)$, $c(t,x)$, $d(t,x)$ and $e(t,x)$ at those possible lines takes the following form

$$
\aligned
a(t,x)=\left\{
\begin{array}{lll}
&&1+j(t,x)(\overline{w}_x^2+\frac{4\overline{w}_x}{T-t}),\quad when \quad\frac{x}{T-t}=0,\\
&&1+j(t,x)(\overline{w}_x^2+\frac{2\overline{w}_x(1+\sqrt{1+(T-t)^2})}{T-t}),\quad when \quad x=-1+\sqrt{1+(T-t)^2},\\
&&1+j(t,x)(\overline{w}_x^2+\frac{4\overline{w}_x}{(1-\delta^2)(T-t)}),\quad\quad when \quad \frac{x}{T-t}=\delta,
\end{array}
\right.
\endaligned
$$
$$
\aligned
b(t,x)=\left\{
\begin{array}{lll}
&&\frac{(T-t)^2}{4+(T-t)^2}+j(t,x)(\overline{w}_x^2+\frac{4\overline{w}_x}{T-t}),\quad when \quad\frac{x}{T-t}=0,\\
&&j(t,x)(\overline{w}_x^2+\frac{2\overline{w}_x(1+\sqrt{1+(T-t)^2})}{T-t}),\quad when \quad x=-1+\sqrt{1+(T-t)^2},\\
&&-\frac{4\delta^2-(1-\delta^2)^2(T-t)^2}{4+(1-\delta^2)^2(T-t)^2}+j(t,x)(\overline{w}_x^2+\frac{4\overline{w}_x}{(1-\delta^2)(T-t)}),\quad\quad when \quad \frac{x}{T-t}=\delta,
\end{array}
\right.
\endaligned
$$
$$
\aligned
c(t,x)=\left\{
\begin{array}{lll}
&&\frac{8}{(T-t)^3+4(T-t)}-j(t,x)\Big(2(\overline{w}_t+\overline{w}_x)\overline{w}_{xx}+\frac{4\overline{w}_{tx}}{T-t}+\frac{2\overline{w}_x}{(T-t)^2}\Big),\quad when \quad\frac{x}{T-t}=0,\\
&&\frac{2}{(\sqrt{1+(T-t)^2}-1)^2+1}-j(t,x)\Big(2\overline{w}_t(\frac{4(\sqrt{1+(T-t)^2}+1)}{(T-t)^2}+\overline{w}_{xx})\\
&&+\frac{2\overline{w}_{tx}(\sqrt{1+(T-t)^2}+1)}{(T-t)}-2(\overline{w}_{xx}-\overline{w}_x\overline{w}_{tx})
+\frac{\overline{w}_x((T-t)^2+\sqrt{1+(T-t)^2}+1)(\sqrt{(T-t)+1}+1)^2}{(T-t)^2}\Big),\\
&&\quad\quad\quad when \quad x=-1+\sqrt{1+(T-t)^2},\\
&&\frac{8}{(1-\delta^2)^2(T-t)^3+4(T-t)}-j(t,x)\Big(2\overline{w}_t(\frac{4\delta}{(1-\delta^2)^2(T-t)^2}+\overline{w}_{xx})+\frac{4\overline{w}_{tx}}{(1-\delta^2)(T-t)}\\
&&+\frac{2\overline{w}_x(1+\delta^2)}{(1-\delta^2)^2(T-t)^2}-2(\frac{2\delta\overline{w}_{xx}}{(1-\delta^2)(T-t)}-\overline{w}_x\overline{w}_{tx})\Big), \quad\quad when \quad \frac{x}{T-t}=\delta,
\end{array}
\right.
\endaligned
$$
$$
\aligned
d(t,x)=\left\{
\begin{array}{lll}
&&j(t,x)\Big(2\overline{w}_{tt}(\overline{w}_x+\frac{2}{T-t})-\frac{4\overline{w}_t}{(T-t)^2}-2\overline{w}_{tx}w_t\Big),\quad when \quad\frac{x}{T-t}=0,\\
&&\frac{2(\sqrt{(T-t)^2+1}-1)}{2+2(T-t)^2-\sqrt{(T-t)^2+1}}+j(t,x)\Big(\frac{2(T-t)(\sqrt{(T-t)^2+1}-1)\overline{w}_x}{2+(T-t)^2-2\sqrt{(T-t)^2+1}}
+2\overline{w}_{tt}(\overline{w}_x+\frac{\sqrt{(T-t)^2+1}+1}{T-t})\\
&&-2\overline{w}_{tx}-\frac{2\overline{w}_t(1+(T-t)^2-\sqrt{(T-t)^2+1})}{2+(T-t)^2-\sqrt{1+(T-t)^2}}-2\overline{w}_{tx}w_t\Big),\quad when \quad x=-1+\sqrt{1+(T-t)^2},\\
&&\frac{8\delta}{(1-\delta^2)^2(T-t)^3+4(T-t)}+j(t,x)\Big(\frac{8\delta\overline{w}_x}{(T-t)^2}
+2\overline{w}_{tt}(\overline{w}_x+\frac{2(T-t)}{(T-t)^2-x^2})\\
&&-\frac{4\delta\overline{w}_{tx}}{(1-\delta^2)(T-t)}-\frac{4(1+\delta^2)\overline{w}_t}{(T-t)^2}-2\overline{w}_{tx}w_t\Big),\quad\quad when \quad \frac{x}{T-t}=\delta,
\end{array}
\right.
\endaligned
$$
$$
\aligned
e(t,x)=\left\{
\begin{array}{lll}
&&4j(t,x)(\frac{1}{T-t}+\overline{w}_x)\overline{w}_t,\quad when \quad\frac{x}{T-t}=0,\\
&&\frac{(T-t)(\sqrt{(T-t)^2+1}-1)}{1+(T-t)^2-\sqrt{(T-t)^2+1}}+2j(t,x)\Big(\frac{\sqrt{(T-t)^2+1}+1}{T-t}\overline{w}_t+\overline{w}_x+2\overline{w}_x\overline{w}_t\Big),\\
&&\quad \quad when \quad x=-1+\sqrt{1+(T-t)^2},\\
&&\frac{8\delta}{(1-\delta^2)^2(T-t)^2+4}+2j(t,x)\Big(\frac{2}{(1-\delta^2)(T-t)}\overline{w}_t+\frac{2\delta\overline{w}_x}{(1-\delta^2)(T-t)}+2\overline{w}_x\overline{w}_t\Big),\quad\quad when \quad \frac{x}{T-t}=\delta.
\end{array}
\right.
\endaligned
$$

Define
$$
\Bcal_R:=\{w\in\Ccal^s_{2}\quad | \quad\|w\|_{\Ccal^s_{2}}\leq R,\quad for~fixed~constants~R\in(0,1)~and~s\geq2\}.
$$

\begin{lemma}
Let $w\in\Bcal_R$. For $(t,x)\in\Bcal_t$, coefficients $a(t,x)$, $b(t,x)$, $c(t,x)$, $d(t,x)$, $e(t,x)$ and $j(t,x)$ of linearized equation (\ref{E3-2}) satisfy
\bel{E3-3}
\aligned
&|a(t,x)|\lesssim(1+\frac{1}{T-t})(1+|\overline{w}_x|+|\overline{w}_x|^2),\\
&|b(t,x)|\lesssim(1+\frac{1}{T-t})(1+|\overline{w}_t|+|\overline{w}_t|^2),\\
&|c(t,x)|\lesssim(1+\frac{1}{(T-t)^2})(1+|\overline{w}_t|+|\overline{w}_x|+|\overline{w}_{tx}|+|\overline{w}_{xx}|+|\overline{w}_x|^2+|\overline{w}_{xx}|^2),\\
& |d(t,x)|\lesssim1+\frac{1}{(T-t)^3}(1+|\overline{w}_x|+|\overline{w}_{tt}|+|\overline{w}_{tx}|+|\overline{w}_t|^2+|\overline{w}_x|^2+|\overline{w}_{tt}|^2+|\overline{w}_{tx}|^2),\\
&|e(t,x)|\lesssim 1+\frac{1}{T-t}(1+|w_t|+|w_x|+|w_t|^2+|w_x|^2),\quad |j(t,x)|\lesssim C,
\endaligned
\ee
where $C$ stands for a positive constant.
\end{lemma}
\begin{proof}
We notice that singular term is $\frac{1}{T-t}$ when $t\rightarrow T^-$.
Observation from the form of coefficients $a(t,x)$, $b(t,x)$, $c(t,x)$, $d(t,x)$, $e(t,x)$ and $j(t,x)$, one can see there is no singular term in $j(t,x)$, and
the highest order of singular terms in $a(t,x)$, $b(t,x)$, $c(t,x)$, $d(t,x)$ and $e(t,x)$ are $\frac{1}{T-t}$, $\frac{1}{T-t}$,
$\frac{1}{(T-t)^2}$, $\frac{1}{(T-t)^3}$ and $\frac{1}{T-t}$, respectively. Thus those estimates hold.
\end{proof}

\subsection{Local existence of linearized equation in the hyperbolic domain}

We now derive energy estimates in the domain $\Omega_1=[0,-1+\sqrt{1+(T-t)^2}]$. In this domain, equation (\ref{E3-1}) is a hyperbolic equation which is degenerate at $x=-1+\sqrt{1+(T-t)^2}$. We supplement this hyperbolic equation with
 initial data
\bel{E3-5-0}
h(0,x)=h_0,\quad h_t(0,x)=h_1,
\ee
where the boundary $\partial\Omega_1=\Sigma_1\cup\Sigma_2$, and $\Sigma_1:=\{x=0\}$ and $\Sigma_2:=\{x=-1+\sqrt{1+(T-t)^2}\}$. 

By (\ref{E03-02}), we can set 
\bel{E3-5}
\aligned
&h(t,x)|_{x\in\Sigma_1}=0,\\
&h_t(t,x)|_{x\in\Sigma_2}=0.
\endaligned
\ee

Obviously, the coefficient $b(t,x)\equiv0$ at $x=-1+\sqrt{1+(T-t)^2}$, thus the linear wave equation (\ref{E3-1}) is degenerate at $x=-1+\sqrt{1+(T-t)^2}$. We follow the idea of Oleinik \cite{Ol} to deal with our case. One can see more applications of it in \cite{Han,N,YAM}.  Let $\overline{\Omega}_1:=[0,-1+\sqrt{1+(T-t)^2})$, then $b(t,x)>0$ in $\overline{\Omega}_1$.

\begin{lemma}
Let positive constant $s\geq2$. Assume that $f(t,x)\in\CC^2([0,\overline{T}];\HH^s(\overline{\Omega}_1))$ and $w\in\Bcal_R$.
Then
the solution $h(t,x)$ of the linearized equation (\ref{E3-2}) with initial data (\ref{E3-5-0}) and boundary condition (\ref{E3-5}) in the domain $\overline{\Omega}_1$ satisfies
$$
\int_{\overline{\Omega}_1}\Big(|h_t|^2+|h_x|^2\Big)dxdt
\lesssim\int_{\overline{\Omega}_1}\Big[|h_1|^2+|(h_0)_x|^2\Big]dx+\int_0^{\overline{T}}\int_{\overline{\Omega}_1}|f|^2dxdt,
\quad \forall t\in[0,\overline{T}].
$$
\end{lemma}
\begin{proof}
Let $\nu$ be a positive constant.
Multiplying both sides of the linearized equation (\ref{E3-2}) with $e^{-\frac{\nu}{(T-t)^2}}h_t$, we get
\bel{E3-5R2}
\aligned
\frac{1}{2}\frac{\del}{\del t}\Big[e^{-\frac{\nu}{(T-t)^2}}a(t,x)|h_t|^2&+e^{-\frac{\nu}{(T-t)^2}}b(t,x)|h_{x}|^2\Big]-e^{-\frac{\nu}{(T-t)^2}}\frac{\del}{\del x}\Big(b(t,x)h_{t}h_x+\frac{1}{2}e(t,x)|h_t|^2\Big)\\
&\quad+e^{-\frac{\nu}{(T-t)^2}}\Big(-c(t,x)-\frac{1}{2}\frac{\del a(t,x)}{\del t}+\frac{2\nu a(t,x)}{(T-t)^3}+\frac{1}{2}\frac{\del e(t,x)}{\del x}\Big)|h_t|^2\\
&\quad+e^{-\frac{\nu}{(T-t)^2}}\frac{1}{2}\Big(-\frac{\del b(t,x)}{\del t}+\frac{2\nu b(t,x)}{(T-t)^3}\Big)|h_x|^2\\
&=-e^{-\frac{\nu}{(T-t)^2}}\Big(d(t,x)+\frac{\del b(t,x)}{\del x}\Big)h_xh_t+f(t,x)e^{-\frac{\nu}{(T-t)^2}}h_t.
\endaligned
\ee

Obviously, all of coefficients in (\ref{E3-5R2}) are bounded when $t$ is away from $T^{-}$. But there is singularity when $t$ is near $T^-$.
So we have to analyze the order of term $(T-t)^{-1}$. On one hand,
direct computations show that
\bel{E3-5R3}
\aligned
\frac{\del j(t,x)}{\del t}&=\frac{8(T-t)\Big(-2(T-t)^2+((T-t)^2-x^2)^2\Big)}{\Big(((T-t)^2-x^2)^2+4(T-t)^2\Big)^2},\\
\frac{\del j(t,x)}{\del x}&=\frac{-8x(T-t)^2((T-t)^2-x^2)}{\Big(((T-t)^2-x^2)^2+4(T-t)^2\Big)^2}.
\endaligned
\ee
So for $t$ near $T^-$, by (\ref{E3-5R3}), it is easy to check that 
$$
\aligned
&\Big|\frac{\del j(t,x)}{\del t}|_{x\in\partial\Omega_1}\Big|\lesssim 1+\frac{1}{T-t},\\
&\Big|\frac{\del j(t,x)}{\del x}|_{x\in\partial\Omega_1}\Big|\lesssim C.
\endaligned
$$
This means that the singular term of $j(t,x)$ in the boundary is at most $(T-t)^{-1}$. Moreover, one can check it is also the singular term inside $\Omega_1$.

On the other hand, it holds
\bel{E3-5R4}
\aligned
\frac{\del a(t,x)}{\del t}&=\frac{\del j(t,x)}{\del t}\Big(\overline{w}_x^2+\frac{4\overline{w}_x(T-t)}{(T-t)^2-x^2}\Big)\\
&+j(t,x)\Big(2\overline{w}_x\overline{w}_{tx}+\frac{4(\overline{w}_{tx}(T-t)-\overline{w}_x)((T-t)^2-x^2)+8(T-t)^2\overline{w}_x}{((T-t)^2-x^2)^2}\Big),\\
\frac{\del b(t,x)}{\del t}&=-\frac{16(T-t)((T-t)^2-x^2)^2}{\Big(((T-t)^2-x^2)^2+4(T-t)^2\Big)^2}-\frac{\del j(t,x)}{\del t}\Big(\overline{w}_t^2+\frac{4\overline{w}_t(T-t)}{(T-t)^2-x^2}\Big)\\
&-2j(t,x)\Big(\overline{w}_t\overline{w}_{tt}+\frac{2(\overline{w}_{tt}(T-t)-\overline{w}_t)((T-t)^2-x^2)+4(T-t)^2\overline{w}_t}{((T-t)^2-x^2)^2}\Big),\\
\frac{\del b(t,x)}{\del x}&=\frac{-8x\Big((T-t)^2((T-t)^2+2x^2-4)-3x^4\Big)}{\Big(((T-t)^2-x^2)^2+4(T-t)^2\Big)^2}\\
&-\frac{\del j(t,x)}{\del x}\Big(2\overline{w}_t\overline{w}_{tx}+\frac{4\overline{w}_{tx}(T-t)((T-t)^2-x^2)+16x(T-t)\overline{w}_t}{((T-t)^2-x^2)^2}\Big),
\endaligned
\ee
and
$$
\aligned
\frac{\del e(t,x)}{\del x}&=\frac{8\Big((T-t)((T-t)^2-x^2)((T-t)^2+3x^2)+32(T-t)^3\Big)}{\Big(((T-t)^2-x^2)^2+4(T-t)^2\Big)^2}\\
&+4\frac{\del j(t,x)}{\del x}\Big(\frac{(T-t)}{(T-t)^2-x^2}\overline{w}_t+\frac{x\overline{w}_x}{(T-t)^2-x^2}+\overline{w}_t\overline{w}_x\Big)\\
&+4j(t,x)\Big(\frac{2x(T-t)\overline{w}_t+(\overline{w}_x+x\overline{w}_{xx})((T-t)^2-x^2)+2x^2\overline{w}_x}{((T-t)^2-x^2)^2}\\
&+\frac{(T-t)\overline{w}_{tx}}{(T-t)^2-x^2}+(\overline{w}_{xx}\overline{w}_t+\overline{w}_x\overline{w}_{tx})\Big).
\endaligned
$$
Note that singualr terms may be higher at the boundary than inside of domain. Thus we analysize the boundary case. Using (\ref{E3-5R4}), we derive
\bel{E3-5R5}
\aligned
&\Big|\frac{\del a(t,x)}{\del t}|_{x\in\partial\Omega_1}\Big|\lesssim \Big(1+\frac{1}{(T-t)^2}\Big)\Big(1+|\overline{w}_x|+|\overline{w}_{tx}|+|\overline{w}_x|^2+|\overline{w}_{tx}|^2\Big),\\
&\Big|\frac{\del b(t,x)}{\del t}|_{x\in\partial\Omega_1}\Big|\lesssim \Big(1+\frac{1}{(T-t)^2}\Big)\Big(|\overline{w}_t|+|\overline{w}_{tt}|+|\overline{w}_t|^2+|\overline{w}_{tt}|^2\Big),\\
&\Big|\frac{\del b(t,x)}{\del x}|_{x\in\partial\Omega_1}\Big|\lesssim \Big(1+\frac{1}{(T-t)^2}\Big)\Big(1+|\overline{w}_t|+|\overline{w}_{tx}|+|\overline{w}_t|^2+|\overline{w}_{tx}|^2\Big),\\
&\Big|\frac{\del e(t,x)}{\del x}|_{x\in\partial\Omega_1}\Big|\lesssim \Big(1+\frac{1}{(T-t)^2}\Big)\Big(1+|\overline{w}_t|+|\overline{w}_x|+|\overline{w}_{xx}|+|\overline{w}_t|^2+|\overline{w}_x|^2+|\overline{w}_{tx}|^2+|\overline{w}_{xx}|^2\Big),
\endaligned
\ee
so by the same proof with Lemma 3.1, inequalities (\ref{E3-5R5}) also hold in $x\in\Omega_1$.

Furthermore, we use Young's inequality to derive
$$
\aligned
&\Big|(d(t,x)+\frac{\del b(t,x)}{\del x})h_xh_t\Big|\lesssim\frac{1}{2}\Big(|d(t,x)|+|\frac{\del b(t,x)}{\del x}|\Big)\Big(|h_x|^2+|h_t|^2\Big),\\
&\Big|f(t,x)h_t\Big|\lesssim\frac{1}{2}\Big(|h_t|^2+|f|^2\Big).
\endaligned
$$

Thus applying above two inequalities to (\ref{E3-5R2}), it holds
\bel{E3-5R6}
\aligned
&\frac{\del}{\del t}\Big[e^{-\frac{\nu}{(T-t)^2}}a(t,x)|h_t|^2+e^{-\frac{\nu}{(T-t)^2}}b(t,x)|h_{x}|^2\Big]-e^{-\frac{\nu}{(T-t)^2}}\frac{\del}{\del x}\Big(2b(t,x)h_{t}h_x+e(t,x)|h_t|^2\Big)\\
&\quad+e^{-\frac{\nu}{(T-t)^2}}\Big(-2c(t,x)-2\frac{\del a(t,x)}{\del t}+\frac{2\nu a(t,x)}{(T-t)^3}\\
&\quad+\frac{\del e(t,x)}{\del x}-|d(t,x)|-|\frac{\del b(t,x)}{\del x}|-1\Big)|h_t|^2\\
&\quad+e^{-\frac{\nu}{(T-t)^2}}\Big(-\frac{\del b(t,x)}{\del t}+\frac{2\nu b(t,x)}{(T-t)^3}-|d(t,x)|-|\frac{\del b(t,x)}{\del x}|\Big)|h_x|^2\\
&\lesssim e^{-\frac{\nu}{(T-t)^2}}|f|^2,
\endaligned
\ee

Note that $w\in\Bcal_R$.
 For a suitable small $R>0$ and a sufficient big $\nu>0$, by (\ref{E3-2x}), (\ref{E3-5R4})-(\ref{E3-5R5}), leading terms of
\bel{E3-5R7}
-2c(t,x)-2\frac{\del a(t,x)}{\del t}+\frac{2\nu a(t,x)}{(T-t)^3}+\frac{\del e(t,x)}{\del x}-|d(t,x)|-|\frac{\del b(t,x)}{\del x}|-1
\ee
and
\bel{E3-5R8}
-\frac{\del b(t,x)}{\del t}+\frac{2\nu b(t,x)}{(T-t)^3}-|d(t,x)|-|\frac{\del b(t,x)}{\del x}|
\ee
are 
$$\frac{2\nu}{(T-t)^3} \quad in \quad \frac{2\nu a(t,x)}{(T-t)^3}$$ 
and 
$$\frac{16(T-t)((T-t)^2-x^2)^2}{\Big(((T-t)^2-x^2)^2+4(T-t)^2\Big)^2}+\frac{((T-t)^2-x^2)^2-4x^2}{((T-t)^2-x^2)^2+4(T-t)^2}\cdot\frac{2\nu}{(T-t)^3}\quad in \quad -\frac{\del b(t,x)}{\del t}+\frac{2\nu b(t,x)}{(T-t)^3},$$ 
respectively.

Here we should notice that all terms containing $\overline{w}$ in $a(t,x)$, $b(t,x)$, $c(t,x)$, $d(t,x)$ and $e(t,x)$ are controlled by a positive constant $\frac{CR}{(T-t)^p}$ with constant $p\geq3$.
Those two leading terms are positive, so for a sufficient big $\nu>0$ can make (\ref{E3-5R7}) and (\ref{E3-5R8}) positive. 
Moreover,  for a sufficient small $R\ll1$ and sufficient big $\mu$ and $\nu$, it holds

\bel{E3-5R9}
\aligned
-2c(t,x)-2\frac{\del a(t,x)}{\del t}+\frac{2\nu a(t,x)}{(T-t)^3}&+\frac{\del e(t,x)}{\del x}-|d(t,x)|-|\frac{\del b(t,x)}{\del x}|-1\\
&>\frac{C}{(T-t)^3}-\frac{CR}{(T-t)^p}>C_{R,\nu}>0,
\endaligned
\ee
and
\bel{E3-5R10}
-\frac{\del b(t,x)}{\del t}+\frac{2\nu b(t,x)}{(T-t)^3}-|d(t,x)|-|\frac{\del b(t,x)}{\del x}|>\frac{C}{(T-t)^3}-\frac{CR}{(T-t)^p}>C_{R,\nu}>0,
\ee
where $C_{R,\nu}$ is a positive constant depending on $R$ and $\nu$.

Thus by noticing (\ref{E3-5R9})-(\ref{E3-5R10}), inequality (\ref{E3-5R6}) leads to 
\bel{E3-5R11}
\aligned
\frac{\del}{\del t}\Big[a(t,x)|h_t|^2+b(t,x)|h_{x}|^2\Big]&-\frac{\del}{\del x}\Big(2b(t,x)h_{t}h_x+e(t,x)|h_t|^2\Big)\\
&+C_{R,\nu}\Big(|h_t|^2+|h_x|^2\Big)\lesssim |f|^2,\quad \forall t\in[0,\overline{T}].
\endaligned
\ee

Integrating (\ref{E3-5R11}) over $\overline{\Omega}_1$, and using the boundary condition (\ref{E3-5}), it holds
$$
\frac{\del}{\del t}\int_{\overline{\Omega}_1}\Big[a(t,x)|h_t|^2+b(t,x)|h_{x}|^2\Big]+C_{R,\nu}\int_{\overline{\Omega}_1}\Big(|h_t|^2+|h_x|^2\Big)\lesssim|f|^2,
$$
so integrating above inequality over $[0,t]$ with $t\in(0,\overline{T}]$, and noticing that $a(t,x)$ and $b(t,x)$ are bounded in $[0,\overline{T}]\times\overline{\Omega}_1$, it holds
$$
\int_{\overline{\Omega}_1}\Big(|h_t|^2+|h_x|^2\Big)dxdt\lesssim\int_{\overline{\Omega}_1}\Big[|h_1|^2+|(h_0)_x|^2\Big]dx+\int_0^{\overline{T}}\int_{\overline{\Omega}_1}|f|^2dxdt.
$$
\end{proof}

In order to obtain higher order energy estimates, we consider the equation of the $x$-derivatives of $h$. For a fixed $2\leq k\leq s$, 
applying $\partial^{k+1}=\partial_t\partial_x^k$ to both sides of (\ref{E3-2}) to get
\bel{E3-6}
a(t,x)\partial_{tt}\partial^{k+1}h-b(t,x)\partial_{xx}\partial^{k+1}h-c(t,x)\partial_t\partial^{k+1}h+d(t,x)\partial_x\partial^{k+1}h-e(t,x)\partial_{tx}\partial^{k+1}h=\textbf{f}_k,
\ee
where $k+1=k_1+k_2$ with $1\leq k_1\leq k+1$ and $0\leq k_2\leq k$, and
\bel{E3-6R1}
\aligned
\textbf{f}_k&:=\partial^{k+1}f-\sum_{k+1=k_1+k_2}\partial^{k_1}a(t,x)\partial_{tt}\partial^{k_2}h+\sum_{k+1=k_1+k_2}\partial^{k_1}b(t,x)\partial_{xx}\partial^{k_2}h\\
&\quad+\sum_{k+1=k_1+k_2}\partial^{k_1}c(t,x)\partial_{t}\partial^{k_2}h-\sum_{k+1=k_1+k_2}\partial^{k_1}d(t,x)\partial_{x}\partial^{k_2}h+\sum_{k+1=k_1+k_2}\partial^{k_1}e(t,x)\partial_{tx}\partial^{k_2}h.
\endaligned
\ee

\begin{lemma}
Let positive constant $2\leq k\leq s$. Assume that $f(t,x)\in\CC^2([0,\overline{T}];\HH^k(\overline{\Omega}_1))$ and $w\in\Bcal_R$.
Then
the solution $h(t,x)$ of the linearized equation (\ref{E3-2}) with initial data (\ref{E3-5-0}) and boundary condition (\ref{E3-5}) in the domain $\overline{\Omega}_1$ satisfies
\bel{XX1-2}
\aligned
\int_{\overline{\Omega}_1}(|\del_t\partial^{k+1}h|^2+|\del_x\partial^{k+1}h|^2)dxdt
&\lesssim\int_{\overline{\Omega}_1}\Big[|\partial^{k+1}h_1|^2+|\del_x\partial^{k+1}h_0|^2\Big]dx\\
&+\int_0^{\overline{T}}\int_{\overline{\Omega}_1}|\partial^{k+1}f|^2dxdt,\quad \forall t\in[0,\overline{T}].
\endaligned
\ee
\end{lemma}
\begin{proof}
Let $\nu$ and $\chi$ be two positive constants.
Multiplying both sides of the linearized equation (\ref{E3-6}) with $e^{-\frac{\nu}{(T-t)^{\chi}}}\del_t\del^{k+1}h$, we get
\bel{E3-5RR2}
\aligned
&\frac{1}{2}\frac{\del}{\del t}\Big[e^{-\frac{\nu}{(T-t)^{\chi}}}a(t,x)|\del_t \partial^{k+1}h|^2+e^{-\frac{\nu}{(T-t)^{\chi}}}b(t,x)|\del_x\partial^{k+1}h|^2\Big]\\
&\quad-e^{-\frac{\nu}{(T-t)^{\chi}}}\frac{\del}{\del x}\Big(b(t,x)\del_t\partial^{k+1}h\del_x\partial^{k+1}h+\frac{1}{2}e(t,x)|\del_t\partial^{k+1}h|^2\Big)\\
&\quad+e^{-\frac{\nu}{(T-t)^{\chi}}}\Big(-c(t,x)-\frac{1}{2}\frac{\del a(t,x)}{\del t}+\frac{\chi\nu a(t,x)}{(T-t)^{\chi+1}}+\frac{1}{2}\frac{\del e(t,x)}{\del x}\Big)|\del_t\partial^{k+1}h|^2\\
&\quad+e^{-\frac{\nu}{(T-t)^{\chi}}}\frac{1}{2}\Big(-\frac{\del b(t,x)}{\del t}+\frac{\chi\nu b(t,x)}{(T-t)^{\chi+1}}\Big)|\del_x\partial^{k+1}h|^2\\
&=-e^{-\frac{\nu}{(T-t)^{\chi}}}\Big(d(t,x)+\frac{\del b(t,x)}{\del x}\Big)\del_x\partial^{k+1}h\del_t\partial^{k+1}h+\textbf{f}_ke^{-\frac{\nu}{(T-t)^{\chi}}}\del_t\partial^{k+1}h.
\endaligned
\ee

We now estimate the right hand side of (\ref{E3-5RR2}) one by one. We use Young's inequality and (\ref{E3-6R1}) to derive
\bel{E3-5RR3}
\aligned
&\Big|\Big(d(t,x)+\frac{\del b(t,x)}{\del x}\Big)\del_x\partial^{k+1}h\del_t\partial^{k+1}h\Big|\leq\frac{1}{2}\Big(d(t,x)+\frac{\del b(t,x)}{\del x}\Big)\Big(|\del_x\partial^{k+1}h|^2+|\del_t\partial^{k+1}h|^2\Big),\\
&\Big|\partial^{k+1}f\del_t\partial^{k+1}h\Big|\leq\frac{1}{2}\Big(|\partial^{k+1}f|^2+|\del_t\partial^{k+1}h|^2\Big).
\endaligned
\ee

We notice that coefficients $a(t,x)$, $b(t,x)$, $c(t,x)$, $d(t,x)$ and $e(t,x)$ contain singular term $\frac{1}{T-t}$ with different orders. So the derivatives of them will increase the order of singular terms.
By Lemma 3.1, Young's inequality and Sobolev embedding $\HH^{k+1}\subset\HH^{k}$ with $k\geq2$, 
for $k+1=k_1+k_2$ with $1\leq k_1\leq k+1$ and $0\leq k_2\leq k$, we integrate by part to get
\bel{E3-5RR4}
\aligned
&\Big|\sum_{k+1=k_1+k_2}\partial^{k_1}c(t,x)\partial_{t}\partial^{k_2}h\del_t\partial^{k+1}h\Big|\lesssim(k+1)\Big(1+\frac{1}{(T-t)^{4+k}}\Big)|\del_t\partial^{k+1}h|^2,\\
&\Big|\sum_{k+1=k_1+k_2}\partial^{k_1}d(t,x)\partial_{x}\partial^{k_2}h\del_t\partial^{k+1}h\Big|\lesssim \frac{k+1}{2}\Big(1+\frac{1}{(T-t)^{4+k}}\Big)\Big(|\del_t\partial^{k+1}h|^2+|\partial^{k+1}h|^2\Big),\\
&\Big|\sum_{k+1=k_1+k_2}\partial^{k_1}e(t,x)\partial_{tx}\partial^{k_2}h\del_t\partial^{k+1}h\Big|\lesssim(k+1)\Big(1+\frac{1}{(T-t)^{3+k}}\Big)|\del_t\partial^{k+1}h|^2,\\
&\Big|\sum_{k+1=k_1+k_2}\partial^{k_1}b(t,x)\partial_{xx}\partial^{k_2}h\del_t\partial^{k+1}h\Big|\lesssim(k+1)\Big(1+\frac{1}{(T-t)^{1+k}}\Big)|\del_t\partial^{k+1}h|^2.
\endaligned
\ee

On the other hand, by inequality $x^ae^{-x}\leq(\frac{a}{e})^a$ with $x>0,a>0$, and integration by parts, there exists a postive constant $C_{k,\nu}$ depending on $k$ and $\nu$ such that
\bel{E3-5RR5}
\sum_{k+1=k_1+k_2}\int_0^T\int_{\Omega_1}e^{-\frac{\nu}{(T-t)^{\chi}}}\partial^{k_1}a(t,x)\partial_{tt}\partial^{k_2}h\del_t\partial^{k+1}h\lesssim C_{k,\nu}\int_0^T\int_{\Omega_1}|\del_t\partial^{k+1}h|^2.
\ee

Thus integrating (\ref{E3-5RR2}) over $[0,t]\times\overline{\Omega}_1$ with $t\in(0,\overline{T}]$, and using (\ref{E3-5RR3})-(\ref{E3-5RR5}), it holds
\bel{E3-5RR6}
\aligned
&\int_{\overline{\Omega}_1}e^{-\frac{\nu}{(T-t)^{\chi}}}\Big(-c(t,x)-\frac{1}{2}\frac{\del a(t,x)}{\del t}+\frac{\chi\nu a(t,x)}{(T-t)^{\chi+1}}+\frac{1}{2}\frac{\del e(t,x)}{\del x}\\
&\quad-4(k+1)(1+\frac{1}{(T-t)^{4+k}})-1\Big)|\del_t\partial^{k+1}h|^2dxdt\\
&\quad+\int_{\overline{\Omega}_1}e^{-\frac{\nu}{(T-t)^{\chi}}}\frac{1}{2}\Big(-\frac{\del b(t,x)}{\del t}+\frac{\chi\nu b(t,x)}{(T-t)^{\chi+1}}-(k+1)(1+\frac{1}{(T-t)^{4+k}})\Big)|\del_x\partial^{k+1}h|^2dxdt\\
&\lesssim\int_{\overline{\Omega}_1}\Big[|\del_t \partial^{k+1}h_0|^2+|\del_x\partial^{k+1}h_0|^2\Big]dx+
\int_0^{\overline{T}}\int_{\overline{\Omega}_1}|\partial^{k+1}f|^2dxdt.
\endaligned
\ee

Note that $w\in\Bcal_R$. So 
for a sufficient big $\nu>4(k+1)$ and $\chi+1\geq k+4$, the term $\frac{C\chi\nu}{(T-t)^{\chi+1}}-\frac{CR}{(T-t)^p}$ with constant $p\geq\chi+1$ is the leading term of 
$$
-c(t,x)-\frac{1}{2}\frac{\del a(t,x)}{\del t}+\frac{\chi\nu a(t,x)}{(T-t)^{\chi+1}}+\frac{1}{2}\frac{\del e(t,x)}{\del x}-4(k+1)(1+\frac{1}{(T-t)^{4+k}})-1,
$$
and
$$
-\frac{\del b(t,x)}{\del t}+\frac{\chi\nu b(t,x)}{(T-t)^{\chi+1}}-(k+1)(1+\frac{1}{(T-t)^{4+k}}).
$$
Furthermore, for a sufficient small $R\ll1$ and sufficient big $\mu$ and $\nu$, it holds
$$
\aligned
&-c(t,x)-\frac{1}{2}\frac{\del a(t,x)}{\del t}+\frac{\chi\nu a(t,x)}{(T-t)^{\chi+1}}+\frac{1}{2}\frac{\del e(t,x)}{\del x}-4(k+1)(1+\frac{1}{(T-t)^{4+k}})-1\\
&\quad\geq\frac{C_{\nu,k}}{(T-t)^{\chi+1}}-\frac{CR}{(T-t)^p}>C_{R,\nu,\chi}>0,\\
&-\frac{\del b(t,x)}{\del t}+\frac{\chi\nu b(t,x)}{(T-t)^{\chi+1}}-(k+1)(1+\frac{1}{(T-t)^{4+k}})\\
&\quad\geq\frac{C_{\nu,k}}{(T-t)^{\chi+1}}-\frac{CR}{(T-t)^p}>C_{R,\nu,\chi}>0,
\endaligned
$$
where $C_{R,\nu,\chi}$ is a positive constant depending on $R$, $\nu$ and $\chi$.

Hence, it follows from (\ref{E3-5RR6}) that
\bel{XX1-3}
\int_{\overline{\Omega}_1}(|\del_t\partial^{k+1}h|^2+|\del_x\partial^{k+1}h|^2)dxdt\lesssim\int_{\overline{\Omega}_1}\Big[|\del_t \partial^{k+1}h_0|^2+|\del_x\partial^{k+1}h_0|^2\Big]dx+\int_0^{\overline{T}}\int_{\overline{\Omega}_1}|\partial^{k+1}f|^2dxdt,
\ee
which means that (\ref{XX1-2}) holds.
\end{proof}

We now follow \cite{Ol}  to give the local existence of solution for linear equation (\ref{E3-2}) in the domain $\Omega_1$.

\begin{lemma}
Let positive constant $\overline{\delta}\ll1$. Assume that $f(t,x)\in\CC^2([0,\overline{T}];\HH^s(\Omega_1))$ and $w\in\Bcal_R$.
Then for a fixed constant $s\geq2$,
equation (\ref{E3-2}) with initial data (\ref{E3-5}) and boundary condition (\ref{E3-5-0}) admits a unique solution
$$
h(t,x)\in\Ccal^s_{2}:=\bigcap_{i= 0}^2\CC^i([0,\overline{T}];\HH^{s-i}(\Omega_1)).
$$
 Moreover, it holds
\bel{E3-7}
\|h(t,x)\|_{\Ccal^s_{2}}\leq\|(h_0,h_1)\|_{\HH^s(\Omega_1)\times\HH^{s-1}(\Omega_1)}+\|f(t,x)\|_{\Ccal^s_{2}}.
\ee
\end{lemma}
\begin{proof}
Let $\kappa$ be a small positve constant. Assume that $f(t,x)$ a compact support in $\Omega_1$.
Consider the approximation equation
\bel{WW01-1}
a(t,x)h_{tt}-b^{\kappa}(t,x)h_{xx}-c(t,x)h_t+d(t,x)h_x-e(t,x)h_{tx}=f(t,x),
\ee
where $(t,x)\in[0,\overline{T}]\times\Omega_1$, and
$$
b^{\kappa}(t,x)=b(t,x)+\kappa>0.
$$

Equation (\ref{WW01-1}) is a strictly linear hyperbolic equation. 
Hence it admits a local $\HH^s$-solution $h^{\kappa}$ with $t\in[0,\overline{T}]$, and it satisfies (\ref{XX1-2}). Meanwhile, by the property of propagation at finite speed, $h^{\kappa}$ is 
of compact support in $\Omega_1$. As in \cite{Ol}, since the right hand side of inequality (\ref{XX1-2}) is independent of $\kappa$, we can
take $h\in\Ccal_2^s$ such that $h^{\kappa}\rightarrow h$ in $\Ccal_2^s$ as $\kappa\rightarrow0$ with a fixed constant $s\geq2$. Therefore, $h$ is the solution of (\ref{E3-2}), and (\ref{XX1-2}) 
remain valid for a limiting function $h$. This gives (\ref{E3-7}).
\end{proof}

\subsection{Local existence of linearized equation in the elliptic domain}

We now consider linear equation (\ref{E3-2}) in the domain 
$$
\Dcal:=\{(t,x)| (t,x)\in[0,\overline{T}]\times\Omega_2\},
$$
where
$$
\Omega_2:=[-1+\sqrt{1+(T-t)^2},\delta(T-t)].
$$
In this case,  equation (\ref{E3-2}) becomes a degenerate linear elliptic-type equation with singular coefficients
as follows
\bel{E3-8}
a(t,x)h_{tt}+\tilde{b}(t,x)h_{xx}-c(t,x)h_t+d(t,x)h_x-e(t,x)h_{tx}=f(t,x),
\ee
where $a(t,x)$, $c(t,x)$, $d(t,x)$ and $e(t,s)$ are given in (\ref{E3-2x}), and
$$
\tilde{b}(t,x):=\frac{4x^2-((T-t)^2-x^2)^2}{((T-t)^2-x^2)^2+4(T-t)^2}-j(t,x)(\overline{w}_t^2+\frac{4\overline{w}_t(T-t)}{(T-t)^2-x^2}),
$$
for $t$ near $T^{-}$, which has the property
\bel{E3-8RR1}
|\tilde{b}(t,x)|\lesssim(1+\frac{1}{T-t})(1+|\overline{w}_t|+|\overline{w}_t|^2),
\ee
and the first term of $\tilde{b}$ satisfies
$$
\frac{4x^2-((T-t)^2-x^2)^2}{((T-t)^2-x^2)^2+4(T-t)^2}>0.
$$

We notice that the elliptic domain is a closed domain.
The degenerate line is at the boundary
$x=-1+\sqrt{1+(T-t)^2}$. The boundary of $\Dcal$ is denoted by
$$
\del\Dcal:=\{t=0\}\cup\Sigma_3\cup\Sigma_4,
$$
where
$$
\aligned
&\Sigma_3:=\{x=-1+\sqrt{1+(T-t)^2}\},\\
&\Sigma_4:=\{x=\delta(T-t)\}.
\endaligned
$$
Moreover, it follows from (\ref{E03-02}) that
\bel{E3-9}
\aligned
&h(t,x)|_{t=0}=h_0,\\
&h_t(t,x)|_{x\in\Sigma_3}=0,\\
&h(t,x)|_{x\in\Sigma_4}=0.
\endaligned
\ee
We remark that the boundary condition (\ref{E3-9}) is a mixed-type boundary condition. Thus all of boundary lines for elliptic domain admit suitable datas. The elliptic equations with mixed-type boundary condition has been studied in \cite{Chen,ChenG, Lieb}. We will solve equation (\ref{E3-8}) with mixed-type boundary condition (\ref{E3-9}). Firstly, we derive some priori estimates of solutions.

\begin{lemma}
Let positive constant $s\geq2$.  Assume that $f(t,x)\in\HH^s(\Dcal)$ and $w\in\Bcal_R$.
Then 
the solution $h(t,x)$ of the linearized equation (\ref{E3-8}) with initial data (\ref{E3-5-0}) and boundary condition  (\ref{E3-9})  in the domain $\Dcal$ satisfies
$$
\int_{\Omega_2}(|h_t|^2+|h_x|^2)dxdt\lesssim\int_{\Omega_2}(|h_1|^2+|(h_0)_x|^2)dx+\int_0^{\overline{T}}\int_{\Omega_2}|f|^2dxdt.
$$
\end{lemma}
\begin{proof}
Let $\nu$ and $\mu$ be two positive constants. Multiplying both sides of the linearized equation (\ref{E3-8}) with ${-\nu \over (T-t)^2}h_t$ and ${\nu\mu \over (T-t)^2}h_x$, respectively, we get
\bel{E3-10}
\aligned
&{\nu \over (T-t)^2}\Big(2c(t,x)+\frac{\del a(t,x)}{\del t}+{2\nu \over T-t} a(t,x)-\frac{\del e(t,x)}{\del x}\Big)|h_t|^2\\
&\quad-{\nu \over (T-t)^2}\Big(\frac{\del \tilde{b}(t,x)}{\del t}+{2\nu\over T-t}\tilde{b}(t,x)\Big)|h_x|^2\\
&\quad-\frac{\del}{\del t}\Big({\nu \over (T-t)^2}a(t,x)|h_t|^2-{\nu \over (T-t)^2}\tilde{b}(t,x)|h_{x}|^2\Big)\\
&\quad-{\nu \over (T-t)^2}\frac{\del}{\del x}\Big(2\tilde{b}(t,x)h_{t}h_x+e(t,x)|h_t|^2\Big)\\
&={2\nu \over (T-t)^2}\Big(d(t,x)+\frac{\del \tilde{b}(t,x)}{\del x}\Big)h_xh_t-2f(t,x){\nu \over (T-t)^2}h_t,
\endaligned
\ee
and
\bel{E3-10RR1}
\aligned
&{\nu\mu \over (T-t)^2}{\del a(t,x) \over \del x}|h_t|^2+{\nu\mu \over (T-t)^2}\Big(-{\del \tilde{b}(t,x) \over \del x}+2d(t,x)+{\del e(t,x) \over \del t}+{2\nu \over T-t}e(t,x)\Big)|h_x|^2\\
&\quad+{\del \over \del t}\Big({2\nu\mu \over (T-t)^2}a(t,x)h_xh_t-{\nu\mu \over (T-t)^2}e(t,x)|h_x|^2\Big)\\
&\quad-{\del \over \del x}\Big({\nu\mu \over (T-t)^2}a(t,x)|h_t|^2+{\nu\mu \over (T-t)^2}\tilde{b}(t,x)|h_x|^2\Big)-{2\nu\mu \over (T-t)^2}c(t,x)h_th_x\\\quad
&={2\nu\mu \over (T-t)^2}f(t,x)h_x.
\endaligned
\ee
Summing up (\ref{E3-10})-(\ref{E3-10RR1}), it holds
\bel{E3-10RR2}
\aligned
&{\nu \over (T-t)^2}\Big(2c(t,x)+\frac{\del a(t,x)}{\del t}+{2\nu \over T-t} a(t,x)-\frac{\del e(t,x)}{\del x}+\mu{\del a(t,x) \over \del x}\Big)|h_t|^2\\
&\quad+{\nu \over (T-t)^2}\Big(-\mu{\del \tilde{b}(t,x) \over \del x}+2\mu d(t,x)+\mu{\del e(t,x) \over \del t}+{2\nu\mu \over T-t}e(t,x)-\frac{\del \tilde{b}(t,x)}{\del t}-{2\nu\over T-t}\tilde{b}(t,x)\Big)|h_x|^2\\
&\quad+\frac{\del}{\del t}\Big({2\nu \mu\over (T-t)^2}a(t,x)h_xh_t-{\nu \over (T-t)^2}(e(t,x)-\mu\tilde{b}(t,x))|h_x|^2-{\nu \over (T-t)^2}a(t,x)|h_t|^2\Big)\\
&\quad-{\nu \over (T-t)^2}\frac{\del}{\del x}\Big(\mu a(t,x)|h_t|^2+\mu\tilde{b}(t,x)|h_x|^2+2\tilde{b}(t,x)h_{t}h_x+e(t,x)|h_t|^2\Big)\\
&={2\nu \over (T-t)^2}\Big(\mu c(t,x)+d(t,x)+\frac{\del \tilde{b}(t,x)}{\del x}\Big)h_xh_t+f(t,x){2\nu \over (T-t)^2}(\mu h_x-h_t).
\endaligned
\ee

We use Young's inequality to derive
$$
\aligned
&2\Big(\mu c(t,x)+d(t,x)+\frac{\del \tilde{b}(t,x)}{\del x}\Big)h_xh_t\lesssim\Big(\mu|c(t,x)|+|d(t,x)|+|\frac{\del \tilde{b}(t,x)}{\del x}|\Big)(|h_t|^2+|h_x|^2),\\
&2f(t,x)(h_t-\mu h_x)\lesssim |f(t,x)|^2+|h_t|^2+\mu|h_x|^2,
\endaligned
$$
which combining with (\ref{E3-10RR2}) gives that
\bel{E3-11}
\aligned
&{\nu \over (T-t)^2}\Big(2c(t,x)+\frac{\del a(t,x)}{\del t}+{2\nu \over T-t} a(t,x)-\frac{\del e(t,x)}{\del x}+\mu{\del a(t,x) \over \del x}-\mu|c(t,x)|\\
&\quad-|d(t,x)|-|\frac{\del \tilde{b}(t,x)}{\del x}|-1\Big)|h_t|^2\\
&\quad+{\nu \over (T-t)^2}\Big(-\mu{\del \tilde{b}(t,x) \over \del x}+2\mu d(t,x)+\mu{\del e(t,x) \over \del t}+{2\nu \over T-t}e(t,x)-\frac{\del \tilde{b}(t,x)}{\del t}\\
&\quad-{2\nu\over T-t}\tilde{b}(t,x)-\mu|c(t,x)|-|d(t,x)|-|\frac{\del \tilde{b}(t,x)}{\del x}|-\mu\Big)|h_x|^2\\
&\quad+\frac{\del}{\del t}\Big({2\nu\mu \over (T-t)^2}a(t,x)h_xh_t-{\nu \over (T-t)^2}(e(t,x)-\mu\tilde{b}(t,x))|h_x|^2-{\nu \over (T-t)^2}a(t,x)|h_t|^2\Big)\\
&\quad-{\nu \over (T-t)^2}\frac{\del}{\del x}\Big(\mu a(t,x)|h_t|^2+\mu\tilde{b}(t,x)|h_x|^2+2\tilde{b}(t,x)h_{t}h_x+e(t,x)|h_t|^2\Big)\\
&\lesssim {2\nu \over (T-t)^2}|f|^2.
\endaligned
\ee

We now estimate the singular terms of (\ref{E3-10}) as $t$ nearby $T^-$. On one hand,
similar to (\ref{E3-5R5}), it holds
\bel{E3-11R1}
\aligned
&\Big|\frac{\del a(t,x)}{\del t}|_{x\in\Omega_2}\Big|\lesssim \Big(1+\frac{1}{(T-t)^2}\Big)\Big(1+|\overline{w}_x|+|\overline{w}_{tx}|+|\overline{w}_x|^2+|\overline{w}_{tx}|^2\Big),\\
&\Big|\frac{\del \tilde{b}(t,x)}{\del t}|_{x\in\Omega_2}\Big|\lesssim \Big(1+\frac{1}{(T-t)^2}\Big)\Big(|\overline{w}_t|+|\overline{w}_{tt}|+|\overline{w}_t|^2+|\overline{w}_{tt}|^2\Big),
\endaligned
\ee
and
\bel{E3-11R1}
\aligned
&\Big|\frac{\del \tilde{b}(t,x)}{\del x}|_{x\in\Omega_2}\Big|\lesssim \Big(1+\frac{1}{(T-t)^2}\Big)\Big(1+|\overline{w}_t|+|\overline{w}_{tx}|+|\overline{w}_t|^2+|\overline{w}_{tx}|^2\Big),\\
&\Big|\frac{\del e(t,x)}{\del x}|_{x\in\Omega_2}\Big|\lesssim \Big(1+\frac{1}{(T-t)^2}\Big)\Big(1+|\overline{w}_t|+|\overline{w}_x|+|\overline{w}_{xx}|+|\overline{w}_t|^2+|\overline{w}_x|^2+|\overline{w}_{tx}|^2+|\overline{w}_{xx}|^2\Big).
\endaligned
\ee

On the other hand, direct computations give that
\bel{E3-11Rr2}
\aligned
&{\del a(t,x) \over \del x}={\del j(t,x) \over \del x}(\overline{w}_x^2+\frac{4\overline{w}_x(T-t)}{(T-t)^2-x^2})\\
&\quad+j(t,x)\Big(2\overline{w}_x\overline{w}_{xx}+\frac{4\overline{w}_{xx}(T-t)((T-t)^2-x^2)+8x(T-t)\overline{w}_x}{((T-t)^2-x^2))^2}\Big),
\endaligned
\ee
and
\bel{E3-11Rr3}
\aligned
&{\del e(t,x) \over \del t}={8(T-t)(((T-t)^2-x^2)^2+4(T-t)^2)+16x^2(T-t)((T-t)^2-x^2) \over (((T-t)^2-x^2)^2+4(T-t)^2)^2}\\
&\quad+2{\del j(t,x)\over \del x}\Big(\frac{2(T-t)}{(T-t)^2-x^2}\overline{w}_t+\frac{2x\overline{w}_x}{(T-t)^2-x^2}+2\overline{w}_x\overline{w}_t\Big)\\
&\quad+2j(t,x)\Big({4x(T-t) \over ((T-t)^2-x^2)^2}\overline{w}_t+{2(T-t)\over (T-t)^2-x^2}\overline{w}_{tx}+{2(\overline{w}_x+x\overline{w}_{xx})((T-t)^2-x^2)+4x^2\overline{w}_x \over ((T-t)^2-x^2)^2}\\
&\quad+2(\overline{w}_{xx}\overline{w}_t+\overline{w}_x\overline{w}_{tx})\Big).
\endaligned
\ee
Obviously, there is no term independent of the derivatives of $\overline{w}$ in ${\del a(t,x) \over \del x}$, and the first term of ${\del e(t,x) \over \del t}$ is 
\bel{E3-11Rr4}
\aligned
0&<{8(T-t)(((T-t)^2-x^2)^2+4(T-t)^2)+16x^2(T-t)((T-t)^2-x^2) \over (((T-t)^2-x^2)^2+4(T-t)^2)^2}\\
&\quad={{8\over T-t}((T-t)^2(1-({x\over T-t})^2)^2+4)+{16x^2\over T-t}(1-({x\over T-t})^2) \over ((1-({x\over T-t})^2)^2+4)^2}\lesssim {16\over T-t}.
\endaligned
\ee
Note that $w\in\Bcal_R$. For a sufficient small $R>0$ and a sufficient big $\nu>\mu>0$, by (\ref{E3-2x}), (\ref{E3-3}), (\ref{E3-5R4}), (\ref{E3-8RR1}) and (\ref{E3-11R1})-(\ref{E3-11Rr3}), 
we know that $\frac{2C\nu}{T-t}-\frac{CR}{(T-t)^p}$ and ${C(\mu+\nu) \over T-t}-\frac{CR}{(T-t)^p}$ $(p\geq1)$ are 
the leading terms of
$$
\aligned
2c(t,x)+\frac{\del a(t,x)}{\del t}+{2\nu \over T-t} a(t,x)-\frac{\del e(t,x)}{\del x}&+\mu{\del a(t,x) \over \del x}-\mu|c(t,x)|\\
&-|d(t,x)|-|\frac{\del \tilde{b}(t,x)}{\del x}|-1,
\endaligned
$$
and
$$
\aligned
-\mu{\del \tilde{b}(t,x) \over \del x}+2\mu d(t,x)+\mu{\del e(t,x) \over \del t}&+{2\nu \over T-t}e(t,x)-\frac{\del \tilde{b}(t,x)}{\del t}\\
&-{2\nu\over T-t}\tilde{b}(t,x)-\mu|c(t,x)|-|d(t,x)|-|\frac{\del \tilde{b}(t,x)}{\del x}|-\mu,
\endaligned
$$
respectively.
Thus for a sufficient small $R\ll1$ and sufficient big $\mu$ and $\nu$, it holds
$$
\aligned
2c(t,x)+\frac{\del a(t,x)}{\del t}+{2\nu \over T-t} a(t,x)-\frac{\del e(t,x)}{\del x}&+\mu{\del a(t,x) \over \del x}-\mu|c(t,x)|\\
&-|d(t,x)|-|\frac{\del \tilde{b}(t,x)}{\del x}|-1\\
&>{C\nu\over T-t}-\frac{CR}{(T-t)^p}>C_{R,\mu,\nu},
\endaligned
$$
and
$$
\aligned
-\mu{\del \tilde{b}(t,x) \over \del x}+2\mu d(t,x)+\mu{\del e(t,x) \over \del t}&+{2\nu \over T-t}e(t,x)-\frac{\del \tilde{b}(t,x)}{\del t}\\
&-{2\nu\over T-t}\tilde{b}(t,x)-\mu|c(t,x)|-|d(t,x)|-|\frac{\del \tilde{b}(t,x)}{\del x}|-\mu\\
&>{C(\mu+\nu)\over T-t}-\frac{CR}{(T-t)^p}>C_{R,\mu,\nu},
\endaligned
$$
where $C_{R,\mu,\nu}$ is a positive constant depending on $R$, $\mu$ and $\nu$.

Hence, integrating (\ref{E3-11}) over $\Dcal$, and noticing the boundary condition (\ref{E3-9}) and initial data (\ref{E3-5-0}), we use Gronwall's inequality to derive
$$
\int_{\Omega_2}(|h_t|^2+|h_x|^2)dxdt\lesssim\int_{\Omega_2}(|h_1|^2+|(h_0)_x|^2)dx+\int_0^{\overline{T}}\int_{\Omega_2}|f|^2dxdt.
$$

\end{proof}

Next we derive higher order energy estimates in the elliptic domain $\Dcal$. For a fixed $2\leq k\leq s$, we
apply $\partial^{k+1}=\partial_t\partial_x^k$ to both sides of (\ref{E3-8}) to get
\bel{RE3-8}
a(t,x)\partial_{tt}\partial^{k+1}h+\tilde{b}(t,x)\partial_{xx}\partial^{k+1}h-c(t,x)\partial_t\partial^{k+1}h+d(t,x)\partial_x\partial^{k+1}h-e(t,x)\partial_{tx}\partial^{k+1}h=\textbf{g}_k,
\ee
where $k+1=k_1+k_2$ with $1\leq k_1\leq k+1$ and $0\leq k_2\leq k$, and
\bel{RE3-8-01}
\aligned
\textbf{g}_k&:=\partial^{k+1}f-\sum_{k+1=k_1+k_2}\partial^{k_1}a(t,x)\partial_{tt}\partial^{k_2}h-\sum_{k+1=k_1+k_2}\partial^{k_1}\tilde{b}(t,x)\partial_{xx}\partial^{k_2}h\\
&+\sum_{k+1=k_1+k_2}\partial^{k_1}c(t,x)\partial_{t}\partial^{k_2}h-\sum_{k+1=k_1+k_2}\partial^{k_1}d(t,x)\partial_{x}\partial^{k_2}h+\sum_{k+1=k_1+k_2}\partial^{k_1}e(t,x)\partial_{tx}\partial^{k_2}h.
\endaligned
\ee

\begin{lemma}
Let positive constant $2\leq k\leq s$.  Assume that $f(t,x)\in\HH^s(\Dcal)$ and $w\in\Bcal_R$.
Then
the solution $h(t,x)$ of the linearized equation (\ref{E3-8}) with initial data (\ref{E3-5-0}) and boundary condition  (\ref{E3-9}) in the domain $\Dcal$ satisfies
\bel{XX1-1}
\int_{\Omega_2}(|\del_t\partial^{k+1}h|^2+|\del_x\partial^{k+1}h|^2)dxdt
\lesssim\int_{\Omega_2}\Big[|\partial^{k+1}h_1|^2+|\del_x\partial^{k+1}h_0|^2\Big]dx+\int_0^{\overline{T}}\int_{\Omega_2}|\partial^{k+1}f|^2dxdt.
\ee
\end{lemma}
\begin{proof}
Let $\nu$, $\mu$ and $\chi$ be three positive constants.
Multiplying both sides of the linearized equation (\ref{RE3-8}) with $e^{\frac{-\nu}{(T-t)^{\chi}}}(\del_t\del^{k+1}h+\mu\del_x\del^{k+1}h)$, it holds
\bel{E3-16}
\aligned
&e^{-\frac{\nu}{(T-t)^{\chi}}}\Big(-2c(t,x)-\frac{\del a(t,x)}{\del t}+{\chi\nu \over (T-t)^{\chi+1}} a(t,x)+\frac{\del e(t,x)}{\del x}+\mu{\del a(t,x) \over \del x}\Big)|\del_t\del^{k+1}h|^2\\
&\quad+e^{-\frac{\nu}{(T-t)^{\chi}}}\Big(-\mu{\del \tilde{b}(t,x) \over \del x}+2\mu d(t,x)+\mu{\del e(t,x) \over \del t}+{\chi\nu\mu \over (T-t)^{\chi+1}}e(t,x)\\
&\quad+\frac{\del \tilde{b}(t,x)}{\del t}-{\chi\nu\over (T-t)^{\chi+1}}\tilde{b}(t,x)\Big)|\del_x\del^{k+1}h|^2\\
&\quad+\frac{\del}{\del t}\Big(\mu e^{-\frac{\nu}{(T-t)^{\chi}}}a(t,x)\del_t\del^{k+1}h\del_x\del^{k+1}h-e^{-\frac{\nu}{(T-t)^{\chi}}}(\mu e(t,x)+\tilde{b}(t,x))|\del_x\del^{k+1}h|^2\\
&\quad+e^{-\frac{\nu}{(T-t)^{\chi}}}a(t,x)|\del_t\del^{k+1}h|^2\Big)\\
&\quad-e^{-\frac{\nu}{(T-t)^{\chi}}}\frac{\del}{\del x}\Big((\mu a(t,x)+e(t,x))|\del_t\del^{k+1}h|^2-\mu\tilde{b}(t,x)|\del_x\del^{k+1}h|^2+2\tilde{b}(t,x)\del_t\del^{k+1}h\del_x\del^{k+1}h\Big)\\
&=2e^{-\frac{\nu}{(T-t)^{\chi}}}\Big(\mu c(t,x)-d(t,x)-\frac{\del \tilde{b}(t,x)}{\del x}\Big)\del_x\del^{k+1}h\del_t\del^{k+1}h+2e^{-\frac{\nu}{(T-t)^{\chi}}}\textbf{g}_k(\mu\del_x\del^{k+1}h+\del_t\del^{k+1}h).
\endaligned
\ee

Note that $\tilde{b}(t,x)=-b(t,x)$.  Similar to (\ref{E3-5RR3})-(\ref{E3-5RR5}), it holds
$$
\aligned
&\Big|\Big(d(t,x)+\frac{\del \tilde{b}(t,x)}{\del x}\Big)\del_x\partial^{k+1}h\del_t\partial^{k+1}h\Big|\leq\frac{1}{2}\Big(d(t,x)+\frac{\del \tilde{b}(t,x)}{\del x}\Big)\Big(|\del_x\partial^{k+1}h|^2+|\del_t\partial^{k+1}h|^2\Big),\\
&\Big|\partial^{k+1}f\del_t\partial^{k+1}h\Big|\leq\frac{1}{2}\Big(|\partial^{k+1}f|^2+|\del_t\partial^{k+1}h|^2\Big),\\
&\Big|\sum_{k+1=k_1+k_2}\int_0^T\int_{\Omega_1}e^{-\frac{\nu}{(T-t)^{\chi}}}\partial^{k_1}a(t,x)\partial_{tt}\partial^{k_2}h\del_t\partial^{k+1}h\Big|\lesssim C_{k,\nu}\int_0^T\int_{\Omega_1}|\del_t\partial^{k+1}h|^2,\\
&\Big|\sum_{k+1=k_1+k_2}\partial^{k_1}c(t,x)\partial_{t}\partial^{k_2}h\del_t\partial^{k+1}h\Big|\lesssim(k+1)(1+\frac{1}{(T-t)^{4+k}})|\del_t\partial^{k+1}h|^2,\\
&\Big|\sum_{k+1=k_1+k_2}\partial^{k_1}d(t,x)\partial_{x}\partial^{k_2}h\del_t\partial^{k+1}h\Big|\lesssim \frac{k+1}{2}(1+\frac{1}{(T-t)^{4+k}})(|\del_t\partial^{k+1}h|^2+|\partial^{k+1}h|^2),\\
&\Big|\sum_{k+1=k_1+k_2}\partial^{k_1}e(t,x)\partial_{tx}\partial^{k_2}h\del_t\partial^{k+1}h\Big|\lesssim(k+1)(1+\frac{1}{(T-t)^{3+k}})|\del_t\partial^{k+1}h|^2,\\
&\Big|\sum_{k+1=k_1+k_2}\partial^{k_1}\tilde{b}(t,x)\partial_{xx}\partial^{k_2}h\del_t\partial^{k+1}h\Big|\lesssim(k+1)(1+\frac{1}{(T-t)^{1+k}})|\del_t\partial^{k+1}h|^2,
\endaligned
$$
where $k+1=k_1+k_2$ with $1\leq k_1\leq k+1$ and $0\leq k_2\leq k$.

Thus by (\ref{E3-16}), it holds
\bel{E3-17}
\aligned
&e^{-\frac{\nu}{(T-t)^{\chi}}}\Big(-2c(t,x)-\frac{\del a(t,x)}{\del t}+{\chi\nu \over (T-t)^{\chi+1}} a(t,x)+\frac{\del e(t,x)}{\del x}+\mu{\del a(t,x) \over \del x}\\
&\quad-\mu |c(t,x)|-|d(t,x)|-|\frac{\del \tilde{b}(t,x)}{\del x}|-4(k+1)(1+\frac{1}{(T-t)^{4+k}})-1\Big)|\del_t\del^{k+1}h|^2\\
&\quad+e^{-\frac{\nu}{(T-t)^{\chi}}}\Big(-\mu{\del \tilde{b}(t,x) \over \del x}+2\mu d(t,x)+\mu{\del e(t,x) \over \del t}+{\chi\nu\mu \over (T-t)^{\chi+1}}e(t,x)+\frac{\del \tilde{b}(t,x)}{\del t}\\
&\quad-{\chi\nu\over (T-t)^{\chi+1}}\tilde{b}(t,x)-\mu |c(t,x)|-|d(t,x)|-|\frac{\del \tilde{b}(t,x)}{\del x}|-4(k+1)(1+\frac{1}{(T-t)^{4+k}})-1\Big)|\del_x\del^{k+1}h|^2\\
&\quad+\frac{\del}{\del t}\Big(\mu e^{-\frac{\nu}{(T-t)^{\chi}}}a(t,x)\del_t\del^{k+1}h\del_x\del^{k+1}h-e^{-\frac{\nu}{(T-t)^{\chi}}}(\mu e(t,x)+\tilde{b}(t,x))|\del_x\del^{k+1}h|^2\\
&\quad+e^{-\frac{\nu}{(T-t)^{\chi}}}a(t,x)|\del_t\del^{k+1}h|^2\Big)\\
&\quad-e^{-\frac{\nu}{(T-t)^{\chi}}}\frac{\del}{\del x}\Big((\mu a(t,x)+e(t,x))|\del_t\del^{k+1}h|^2-\mu\tilde{b}(t,x)|\del_x\del^{k+1}h|^2+2\tilde{b}(t,x)\del_t\del^{k+1}h\del_x\del^{k+1}h\Big)\\
&\lesssim e^{-\frac{\nu}{(T-t)^{\chi}}}|\del^{k+1}f|^2.
\endaligned
\ee

Note that $w\in\Bcal_R$. So 
for a sufficient big $\nu>4(k+1)$ and $\chi+1\geq k+4$ with $k\geq2$,  $\frac{\chi\nu}{(T-t)^{\chi+1}}-\frac{R}{(T-t)^p}$ and ${\chi\nu\mu \over (T-t)^{\chi+1}}-\frac{R}{(T-t)^p}$ $(p\geq\chi+1)$ are the leading terms of
$$
\aligned
&-2c(t,x)-\frac{\del a(t,x)}{\del t}+{\chi\nu \over (T-t)^{\chi+1}} a(t,x)+\frac{\del e(t,x)}{\del x}+\mu{\del a(t,x) \over \del x}\\
&\quad-\mu |c(t,x)|-|d(t,x)|-|\frac{\del \tilde{b}(t,x)}{\del x}|-4(k+1)(1+\frac{1}{(T-t)^{4+k}})-1
\endaligned
$$
and
$$
\aligned
&-\mu{\del \tilde{b}(t,x) \over \del x}+2\mu d(t,x)+\mu{\del e(t,x) \over \del t}+{\chi\nu\mu \over (T-t)^{\chi+1}}e(t,x)+\frac{\del \tilde{b}(t,x)}{\del t}\\
&\quad-{\chi\nu\over (T-t)^{\chi+1}}\tilde{b}(t,x)-\mu |c(t,x)|-|d(t,x)|-|\frac{\del \tilde{b}(t,x)}{\del x}|-4(k+1)(1+\frac{1}{(T-t)^{4+k}})-1,
\endaligned
$$
respectively.

Furthermore, for a sufficient small $R\ll1$ and sufficient big $\mu$ and $\nu$, it holds
$$
\aligned
&-2c(t,x)-\frac{\del a(t,x)}{\del t}+{\chi\nu \over (T-t)^{\chi+1}} a(t,x)+\frac{\del e(t,x)}{\del x}+\mu{\del a(t,x) \over \del x}-\mu |c(t,x)|-|d(t,x)|\\
&\quad-|\frac{\del \tilde{b}(t,x)}{\del x}|-4(k+1)(1+\frac{1}{(T-t)^{4+k}})-1\\
&\geq\frac{C\nu\chi}{(T-t)^{\chi+1}}-\frac{RC}{(T-t)^p}\\
&>C_{R,\mu,\nu,\chi}>0,\\
&-\mu{\del \tilde{b}(t,x) \over \del x}+2\mu d(t,x)+\mu{\del e(t,x) \over \del t}+{\chi\nu\mu \over (T-t)^{\chi+1}}e(t,x)+\frac{\del \tilde{b}(t,x)}{\del t}-{\chi\nu\over (T-t)^{\chi+1}}\tilde{b}(t,x)-\mu |c(t,x)|\\
&\quad-|d(t,x)|-|\frac{\del \tilde{b}(t,x)}{\del x}|-4(k+1)(1+\frac{1}{(T-t)^{4+k}})-1\\
&\geq\frac{C\chi\mu\nu}{(T-t)^{\chi+1}}-\frac{RC}{(T-t)^p}\\
&>C_{R,\mu,\nu,\chi}>0,
\endaligned
$$
where $C_{R,\mu,\nu,\chi}$ stands for a positive constant depending on $R$, $\mu$, $\nu$ and $\chi$.

Note that $\tilde{b}(t,x)|_{x\in\Sigma_3}=0$ and $\mu a(t,x)|_{x\in\Sigma_3}>0$. Hence, integrating (\ref{E3-17}) over $\Dcal$, we use Gronwall's inequality to derive
$$
\int_{\Omega_2}(|\del_t\del^{k+1}h|^2+|\del_x\del^{k+1}h|^2)dxdt
\lesssim \int_{\Omega_2}\Big[|\del_t \partial^{k+1}h_0|^2+|\del_x\partial^{k+1}h_0|^2\Big]dx+\int_0^{\overline{T}}\int_{\Omega_2}|\partial^{k+1}f|^2dxdt.
$$
\end{proof}

\begin{lemma}
Let positive constant $s\geq2$. Assume that $f(t,x)\in\HH^s(\Dcal)$ and $w\in\Bcal_R$.
Then
equation (\ref{E3-8}) admits a unique solution 
$
h(t,x)\in\HH^{s}(\Dcal).
$
 Moreover, it holds
\bel{E3-18R1}
\|h(t,x)\|_{\HH^s}\leq\|(h_0,h_1)\|_{\HH^s\times\HH^{s-1}}+\|f(t,x)\|_{\HH^s}.
\ee
\end{lemma}
\begin{proof}
Let $\theta>0$. We consider the regularized operator
\bel{E3-18}
a(t,x)h_{tt}+(\tilde{b}(t,x)+\theta)h_{xx}-c(t,x)h_t+d(t,x)h_x-e(t,x)h_{tx}=f(t,x),\quad (t,x)\in\Dcal,
\ee
which is a uniformly elliptic equation in $\Dcal$. Hence there exists a solution $h_{\theta}\in \HH_0^1(\Dcal)$. Obviously, $(0,0)\not\in\Dcal$.
By the classical theory of uniform elliptic equations, we have $h\in\mathbb{C}^{\infty}(\Dcal)\cap\mathbb{C}(\Dcal)$. Following the idea of \cite{Han,Han1},
we can use a similar proof of process given in Lemma 3.5-3.6 to derive some priori estimates on $h_{\theta}$ independent of $\theta$. Then we take $\theta\rightarrow0$, 
one can get equation (\ref{E3-18})
has a unique solution $h\in\HH^s(\Dcal)$, meanwhile, it satisfies the mixed boundary condition (\ref{E3-9}). Furthermore, estimate (\ref{E3-18R1}) can be obtained by Lemma 3.6.
\end{proof}


\subsection{Estimates of solution for linearized equation}

We will solve the nonlinear equation (\ref{E3-1}) by using Nash-Moser iteration scheme. One can see \cite{H,Moser, Nash,R} for more details on this method. Here we should notice that hyperbolic property or elliptic property of equation (\ref{E3-1}) only depends on
 the sign of coefficient $b(t,x)$ of diffusion term, but independent of $w$. By Lemma 3.4 and Lemma 3.7, we know that the solution in hyperbolic domain and elliptic domain has the same regularity, so we combine with two solutions 
in $\Omega_1$ and $\Dcal$, then giving the solution of equation (\ref{E3-1}) in domain $\Omega=\Omega_1\cup\Dcal$. Moreover, we have the following estimate.

\begin{lemma}
Let positive constant $s\geq2$ . Assume that $f(t,x)\in\CC^2([0,\overline{T}];\HH^s(\Omega))$ and $w\in\Bcal_R$.
Then the solution $h(t,x)$ of the linearized equation (\ref{E3-2}) with the initial data (\ref{E3-5-0}) and boundary condition (\ref{E03-02}) in the domain $\Omega$ satisfies
\bel{XX1-5}
\|h(t,x)\|_{\Ccal^s_{2}}\leq\|(h_0,h_1)\|_{\HH^s\times\HH^{s-1}}+\|f(t,x)\|_{\Ccal^s_{2}}.
\ee
\end{lemma}
\begin{proof}
We notice that $\Sigma_2=\Sigma_3=\{x=-1+\sqrt{1+(T-t)^2}\}$. Summing up (\ref{XX1-3}) with (\ref{XX1-1}), 
it holds
$$
\aligned
&\int_{\Omega}(|\del_t\partial^{s+1}h|^2+|\del_x\partial^{s+1}h|^2)dxdt\\
&=\int_{\Omega_1}(|\del_t\partial^{s+1}h|^2+|\del_x\partial^{s+1}h|^2)dxdt+\int_{\Omega_2}(|\del_t\del^{s+1}h|^2+|\del_x\del^{s+1}h|^2)dxdt\\
&\lesssim\int_{\Omega}\Big[|\del_t \partial^{s+1}h_0|^2+|\del_x\partial^{s+1}h_0|^2\Big]dx+\int_0^{\overline{T}}\int_{\Omega}|\partial^{s+1}f|^2dxdt.
\endaligned
$$
\end{proof}


\subsection{Local existence of solutions for nonlinear equation }

We introduce a family of smooth operators possessing the following properties.
\begin{lemma}(see \cite{Alin})
There is a family $\{\Pi_{\theta}\}_{\theta\geq1}$ of smoothing operators in the space $\mathbb{H}^{s}(\Omega)$ acting on the class of functions such that
\bel{E4-0}
\aligned
&\|\Pi_{\theta}u\|_{\HH^{s_1}}\leq C\theta^{(s_1-s_2)_+}\|u\|_{\HH^{s_2}},~~s_1,~s_2\geq0,\\
&\|\Pi_{\theta}u-u\|_{\HH^{s_1}}\leq C\theta^{s_1-s_2}\|u\|_{\HH^{s_2}},~~0\leq s_1\leq s_2,\\
&\|\frac{d}{d\theta}\Pi_{\theta}u\|_{\HH^{s_1}}\leq C\theta^{s_1-s_2-1}\|u\|_{\HH^{s_2}},~~s_1,~s_2\geq0,
\endaligned
\ee
where $C$ is a positive constant and $(s_1-s_2)_+:=\max(0,s_1-s_2)$. 
\end{lemma}

In our iteration scheme, we set
$$
\theta=N_m=2^m,\quad \forall m= 0,1,2,\ldots,
$$
then by (\ref{E4-0}), it holds
\bel{E4-1}
\|\Pi_{N_m}u\|_{\HH^{s_1}}\lesssim N_m^{s_1-s_2}\|u\|_{\HH^{s_2}},\quad \forall s_1\geq s_2.
\ee

Introduce an auxiliary function
\bel{E4-2}
\psi(t,x)=w(t,x)-\eps w_0(x)-\eps tw_1(x),
\ee
where $(\eps w_0(x),\eps w_1(x))$ given in (\ref{E03-01}) is the small initial data of equation (\ref{E3-1}).

By means of auxiliary function (\ref{E4-2}), we get zero initial data
$$
\psi(0,x)=0,\quad \psi_t(0,x)=0,
$$
meanwhile, we reduce equation (\ref{E3-1}) into a nonlinear equation with zero initial data as follows
\bel{E4-3}
\aligned
\psi_{tt}&-\frac{((T-t)^2-x^2)^2-4x^2}{((T-t)^2-x^2)^2+4(T-t)^2}\psi_{xx}
-\frac{8(T-t)}{((T-t)^2-x^2)^2+4(T-t)^2}\psi_t\\
&+\frac{8x}{((T-t)^2-x^2)^2+4(T-t)^2}\psi_x-\frac{8x(T-t)}{((T-t)^2-x^2)^2+4(T-t)^2}\psi_{tx}\\
&+\frac{((T-t)^2-x^2)^2}{((T-t)^2-x^2)^2+4(T-t)^2}\Big[(\frac{4x(T-t)}{((T-t)^2-x^2)^2}+\psi_{tt})\psi_x^2+(\frac{4x(T-t)}{((T-t)^2-x^2)^2}+\psi_{xx})\psi_t^2\\
&-2(\frac{2(T-t)}{(T-t)^2-x^2}\psi_t+\frac{2x}{(T-t)^2-x^2}\psi_{x})\psi_{tx}
+2(\frac{2(T-t)}{(T-t)^2-x^2}\psi_{tt}-\frac{2((T-t)^2+x^2)}{((T-t)^2-x^2)^2}\psi_{t})\psi_{x}\\
&+2(\frac{2x}{(T-t)^2-x^2}\psi_{xx}-\psi_x\psi_{tt})\psi_t\Big]=F(t,x),
\endaligned
\ee
where
$$
\aligned
F(t,x)&:=\frac{((T-t)^2-x^2)^2-4x^2}{((T-t)^2-x^2)^2+4(T-t)^2}(w_0''+tw_1'')+\frac{8(T-t)}{((T-t)^2-x^2)^2+4(T-t)^2}w_1\\
&-\frac{8x}{((T-t)^2-x^2)^2+4(T-t)^2}(w_0'+tw_1')+\frac{8x(T-t)}{((T-t)^2-x^2)^2+4(T-t)^2}w_1'\\
&-\frac{((T-t)^2-x^2)^2}{((T-t)^2-x^2)^2+4(T-t)^2}\Big[(\frac{4x(T-t)}{((T-t)^2-x^2)^2}(w_0'+tw_1')^2\\
&+(\frac{4x(T-t)}{((T-t)^2-x^2)^2}+w_0''+tw_1'')w_1^2-2(\frac{2(T-t)}{(T-t)^2-x^2}w_1'+\frac{2x}{(T-t)^2-x^2}(w_0'+tw_1'))w_1'\\
&-2\frac{2((T-t)^2+x^2)}{((T-t)^2-x^2)^2}w_1')(w_0'+tw_1')+\frac{4x}{(T-t)^2-x^2}(w_0''+tw_1'')w_1\Big].
\endaligned
$$

We consider the approximation equation of nonlinear equation (\ref{E4-3}) as follows
\bel{E4-4}
\aligned
\Fcal(\psi)&:=\psi_{tt}-\frac{((T-t)^2-x^2)^2-4x^2}{((T-t)^2-x^2)^2+4(T-t)^2}\psi_{xx}
-\frac{8(T-t)}{((T-t)^2-x^2)^2+4(T-t)^2}\psi_t\\
&+\frac{8x}{((T-t)^2-x^2)^2+4(T-t)^2}\psi_x-\frac{8x(T-t)}{((T-t)^2-x^2)^2+4(T-t)^2}\psi_{tx}\\
&+\frac{((T-t)^2-x^2)^2}{((T-t)^2-x^2)^2+4(T-t)^2}\Pi_{N_m}\Big[(\frac{4x(T-t)}{((T-t)^2-x^2)^2}+\psi_{tt})\psi_x^2+(\frac{4x(T-t)}{((T-t)^2-x^2)^2}+\psi_{xx})\psi_t^2\\
&-2(\frac{2(T-t)}{(T-t)^2-x^2}\psi_t+\frac{2x}{(T-t)^2-x^2}\psi_{x})\psi_{tx}
+2(\frac{2(T-t)}{(T-t)^2-x^2}\psi_{tt}-\frac{2((T-t)^2+x^2)}{((T-t)^2-x^2)^2}\psi_{t})\psi_{x}\\
&+2(\frac{2x}{(T-t)^2-x^2}\psi_{xx}-\psi_x\psi_{tt})\psi_t\Big]-F(t,x).
\endaligned
\ee

Assume that the $m$-th approximation solution of (\ref{E4-4}) is denoted by $\psi^{(m)}$ with $m=0,1,2,\ldots$. Let
$$
h^{(m)}:=\psi^{(m)}-\psi^{(m-1)},\quad for \quad m=1,2,\ldots,
$$
so we have
$$
\psi^{(m)}=\psi^{(0)}+\sum_{i=1}^mh^{(i)}.
$$
Our target is to prove that $\psi^{(\infty)}$ is a local solution of nonlinear equation (\ref{E4-3}). It is equivalent to show the series $\sum_{i=1}^mh^{(i)}$ is convergence.

Linearizing nonlinear equation (\ref{E4-4}) around $h^{(m)}$, we get a linearized operator
$$
\Lcal(h^{(m)}):=a(t,x)h^{(m)}_{tt}-b(t,x)h^{(m)}_{xx}-c(t,x)h^{(m)}_t+d(t,x)h^{(m)}_x-e(t,x)h^{(m)}_{tx},
$$
where coefficients $a(t,x)$, $b(t,x)$, $c(t,x)$, $d(t,x)$ and $e(t,x)$ have the same form with (\ref{E3-2x}).

We can choose an initial approximation solution $\psi^{(0)}$ such that initial error term
$$
E^{(0)}:=\Lcal(\psi^{(0)})-F(t,x),\quad \forall t\in[0,\overline{T}],
$$
satisfies
\bel{E4-5}
\aligned
&\psi^{(0)}\neq0,\\
&\|\psi^{(0)}\|_{\Ccal_2^{s_0+2}}\lesssim \eps,\\
&\|E^{(0)}\|_{\Ccal_2^{s_0}}\lesssim\eps,
\endaligned
\ee
for some $s_0\geq2$ and constant $\overline{T}\in[T-\overline{\delta},T)$ with $0<\overline{\delta}\ll1$. It is easy to see that (\ref{E4-5}) holds for a sufficient small $\eps$,
and $\|F(t,x)\|_{\HH^s(\Omega)}$ is bounded in $[0,\overline{T}]$, which can be controlled by $\eps$ due to small initial data $(\eps w_0,\eps w_1)$.

The $m$-th error terms is denoted by
\bel{E4-6}
R(h^{(m)}):=\Fcal(\psi^{(m-1)}+h^{(m)})-\Fcal(\psi^{(m-1)})-\Lcal(h^{(m)}),
\ee
which is also nonlinear term of approximation equation (\ref{E4-4}). The exact form of nonlinear term (\ref{E4-6}) is very complicated, here we do not write it down. 

\begin{lemma}
Let $\psi^{(m)}\in\Bcal_R$. For any $s\geq2$ and $t\in[0,\overline{T}]$, it holds
\bel{E4-9R1}
\|R(h^{(m)})\|_{\Ccal_2^s}\lesssim N_m^4\|h^{(m)}\|^2_{\Ccal_2^{s}}.
\ee
\end{lemma}
\begin{proof}
We notice that the highest order of nonlinear term in (\ref{E4-6}) is $3$, and the highest order of derivatives on $x$ and $t$ in (\ref{E4-6})
are $2$. Since the solution of (\ref{E4-3}) should be constructed in $\Bcal_R$, so we should prove there exists a positive constant $s$ such that
$$
\|h^{(m)}\|_{\Ccal_2^{s}}\leq R<1,\quad \forall m\in\NN,
$$ 
thus it holds
$$
\|h^{(m)}\|^p_{\Ccal_2^{s}}\leq \|h^{(m)}\|^2_{\Ccal_2^{s}},\quad for\quad p\geq2.
$$

We use (\ref{E4-1}) and Young's inequality to get
$$
\|R(h^{(m)})\|_{\Ccal_2^s}\lesssim \|h^{(m)}\|^2_{\Ccal_2^{s+2}}\lesssim N_m^4\|h^{(m)}\|^2_{\Ccal_2^{s}}.
$$

\end{proof}

The following Lemma is to construct the $m$-th approximation solution.
\begin{lemma}
Let $s\geq2$. The approximation equation (\ref{E4-4}) 
admits a solution $h^{(m)}\in\Ccal_2^{s}$ satisfying
\bel{E4-10R1}
\|h^{(m)}\|_{\Ccal_2^{s}}\lesssim \|E^{(m-1)}\|_{\Ccal_2^{s}},\quad t\in[0,\overline{T}],
\ee
where the error term
\bel{E4-10R2}
E^{(m-1)}:=\Fcal(\psi^{(m-1)})=R(h^{(m-1)}).
\ee
\end{lemma}
\begin{proof}
We notice that the initial approximation solution $\psi^{(0)}$ satisfies (\ref{E4-5}). The $m-1$-th approximation solution is 
$$
\psi^{(m-1)}=\psi^{(0)}+\sum_{i=1}^{m-1}h^{(i)}.
$$
Then we will find the $m$-th approximation solution $\psi^{(m)}$, which is equivalent to find $h^{(m)}$ such that 
\bel{E4-7}
\psi^{(m)}=\psi^{(m-1)}+h^{(m)}.
\ee
Substituting (\ref{E4-7}) into (\ref{E4-4}), it holds
$$
\Fcal(\psi^{(m)})=\Fcal(\psi^{(m-1)})+\Lcal(h^{(m)})+R_m(h^{(m)}).
$$

Let 
$$
\Fcal(\psi^{(m-1)})+\Lcal(h^{(m)})=0,
$$
with zero initial data
$$
h^{(m)}(0,x)=0,\quad h^{(m)}_t(0,x)=0.
$$
By Lemma 3.8, above zero initial data problem admits a solution $h^{(m)}\in\Ccal_2^{s}$. Furthermore, by (\ref{XX1-5}), it holds
$$
\|h^{(m)}\|_{\Ccal_2^{s}}\lesssim \|\Fcal(\psi^{(m-1)})\|_{\Ccal_2^{s}}.
$$

Moreover, it holds
$$
E^{(m)}:=\Fcal(\psi^{(m)})=R(h^{(m)}).
$$
\end{proof}

For a fixed constant $s\geq2$, let $2\leq \bar{s}<s_0\leq s$, and
$$
\aligned
&s_l:=\bar{s}+\frac{s-\bar{s}}{2^l},\\
&\alpha_{l+1}:=s_l-s_{l+1}=\frac{s-\bar{s}}{2^{l+1}},
\endaligned
$$
which gives that
\bel{EX1-1}
s_0>s_1>\ldots>s_l>s_{l+1}>\ldots.
\ee

\begin{proposition}
Let $\overline{\delta}$ and $\eps$ be two small positive constants.
Equation (\ref{E3-1}) with small initial data (\ref{E03-01}) admits a solution 
$$
w(t,x)=\psi^{(\infty)}(t,x)+\eps w_0(x)+\eps tw_1(x),\quad \forall t\in[0,\overline{T}],
$$
where $\overline{T}\in[T-\overline{\delta},T)$ is a positive constant, and for a fixed $\bar{s}\geq2$,
$$
\psi^{(\infty)}(t,x)=\psi^{(0)}+\sum_{m=1}^{\infty}h^{(m)}\in\Ccal_2^{\bar{s}}.
$$
\end{proposition}
\begin{proof}
The proof is based on the induction. Note that $N_m=N_0^m$ with $N_0>1$. $\forall m=1,2,\ldots$, we claim that for any $t\in[0,\overline{T}]$ with $\overline{T}\in[T-\overline{\delta},T)$, there exists a sufficient small positive constant $d$ such that
\bel{E4-12}
\aligned
&\|h^{(m)}\|_{\Ccal_2^{s_m}}<d^{2^m},\\
&\|E^{(m-1)}\|_{\Ccal_2^{s_m}}<d^{2^{m+1}},\\
&\psi^{(m)}\in\Bcal_R.
\endaligned
\ee

For the case of $m=1$, we recall that the assumption on initial approximation (\ref{E4-5}), i.e. $\forall t\in[0,\overline{T}]$,
$$
\aligned
&\psi^{(0)}\neq0,\\
&\|\psi^{(0)}\|_{\Ccal_2^{s_0+2}}\lesssim \eps,\\
&\|E^{(0)}\|_{\Ccal_2^{s_0}}\lesssim\eps.
\endaligned
$$
So by (\ref{E4-10R1}), let $0<\eps<N_0^{-8}d^2<R\ll1$, it holds
$$
\|h^{(1)}\|_{\Ccal_2^{s_1}}\lesssim \|E^{(0)}\|_{\Ccal_2^{s_0}}\lesssim \eps<d.
$$
Moreover, by (\ref{E4-9R1}) and (\ref{E4-10R2}), we derive
$$
\|E^{(1)}\|_{\Ccal_2^{s_1}}\lesssim\|R_1(h^{(1)})\|_{\Ccal_2^{s_1}}\lesssim N_1^4\|h^{(1)}\|^2_{\Ccal_2^{s_1}}\lesssim \eps N_1^4<d^2,
$$
and 
$$
\|\psi^{(1)}\|_{\Ccal_2^{s_1}}\lesssim\|\psi^{(0)}\|_{\Ccal_2^{s_1}}+\|h^{(1)}\|_{\Ccal_2^{s_1}}\lesssim 2\eps<R,
$$
which means that $\psi^{(1)}\in\Bcal_R$.

Assume that the case $m-1$ step holds, i.e.
\bel{E4-13}
\aligned
&\|h^{(m-1)}\|_{\Ccal_2^{s_{m-1}}}<d^{2^{m-1}},\\
&\|E^{(m-1)}\|_{\Ccal_2^{s_{m-1}}}<d^{2^{m}},\\
&\psi^{(m-1)}\in\Bcal_R,
\endaligned
\ee
then we prove the case $m$ step holds. By (\ref{E4-10R1}) and (\ref{E4-13}), we derive
\bel{E4-14R1}
\|h^{(m)}\|_{\Ccal_2^{s_m}}\lesssim \|E^{(m-1)}\|_{\Ccal_2^{s_{m}}}< \|E^{(m-1)}\|_{\Ccal_2^{s_{m-1}}}<d^{2^{m}},
\ee
which combining with (\ref{E4-9R1}), (\ref{E4-10R2}) and (\ref{EX1-1}), it holds
\bel{E4-14}
\aligned
\|E^{(m)}\|_{\Ccal_2^{s_{m}}}&=\|R(h^{(m)})\|_{\Ccal_2^{s_{m}}}\\
&\lesssim N_{m-1}^4\|E^{(m-1)}\|^2_{\Ccal_2^{s_{m-1}}}\\
&\lesssim N_0^{4(m-1)+8(m-2)}\|E^{(m-2)}\|^{2^2}_{\Ccal_2^{s_{m-2}}}\\
&\lesssim \ldots,\\
&\lesssim (N_0^8\|E_0\|_{\Ccal_2^{s_0}})^{2^m}.
\endaligned
\ee

So by (\ref{E4-5}), we can choose a sufficient small positive constant $\eps$ such that
$$
0<N_0^8\|E_0\|_{\Ccal_2^{s_0}}<N_0^8\eps<d^2,
$$
which combining with (\ref{E4-14}) gives that 
$$
\|E^{(m)}\|_{\Ccal_2^{s_{m}}}<d^{2^{m+1}}.
$$
On the other hand, by (\ref{E4-14R1}), it holds
$$
\|\psi^{(m)}\|_{\Ccal_2^{s_{m}}}\lesssim \|\psi^{(m-1)}\|_{\Ccal_2^{s_{m-1}}}+\|h^{(m)}\|_{\Ccal_2^{s_{m}}}\lesssim \eps+\sum_{i=1}^md^{2^i}<R.
$$
This means that $\psi^{(m)}\in\Bcal_R$. Hence, we conclude that (\ref{E4-12}) holds.

Furthermore, it follows from (\ref{E4-12}) that
$$
\lim_{m\rightarrow\infty}\|E^{(m)}\|_{\Ccal_2^{s_{m}}}=0.
$$ 

In conclusion, equation (\ref{E4-3}) admits a solution 
$$
\psi^{(\infty)}=\psi^{(0)}+\sum_{m=1}^{\infty}h^{(m)}\in\Ccal_2^{\bar{s}}.
$$
Furthermore, by (\ref{E4-2}), equation (\ref{E3-1}) with small initial data (\ref{E03-01}) admits a solution 
$$
w(t,x)=\psi^{(\infty)}(t,x)+\eps w_0(x)+\eps tw_1(x),\quad \forall t\in[0,\overline{T}].
$$

\end{proof}

\subsection{Proof of Theorem 1.2.}
Since initial data is small, i.e. for a sufficient small $\eps>0$, it holds
$$
\|u_0(x)-u_k(0,x)\|_{\HH^{s}(\Bcal_0)}+\|u_1(x)-\del_tu_k(0,x)\|_{\HH^{s}(\Bcal_0)}<\eps,
$$
then by Proposition 3.1, for a fixed constant $\bar{s}\geq2$, equation (\ref{E1-1}) admits a local solution $u(t,x)$ such that $u(t,x)\in\Ccal_2^{\bar{s}}$ and
$$
\| u(t,x) -u_k(t,x)\|_{\Ccal_2^{\bar{s}}}=\|\psi^{(\infty)}(t,x)+\eps w_0(x)+\eps tw_1(x)\|_{\Ccal_2^{\bar{s}}}
< \eps, 
\qquad \forall t\in[0,\overline{T}].
$$
Therefore, self-similar solutions $u_k(t,x)=k\ln(\frac{T-t+x}{T-t-x})$ $(\forall k\in\RR/\{0\})$ are Lyapunov nonlinear stability in the domain $\Bcal_{t}:=\Big\{(t,x)\Big| x\in[0,\delta(T-t)],\quad t\in[0,\overline{T}]\Big\}$ with $\overline{T}\in[T-\delta,T)$ and $0<\delta\ll1$.\\
\\



\begin{thebibliography}{xx}

\bibitem{Ale}
M. Alejo, C. Mu$\tilde{n}$oz, 
Almost sharp nonlinear scattering in one dimensional Born-Infeld equations arising in nonlinear electrodynamics.
Proc. Amer. Math. Soc. 146 (2018) 2225-2237.













\bibitem{Alin}
S. Alinhac, Existence d'ondes de rar\'{e}faction pour des syst$\grave{e}$mes quasi-lin\'{e}aires hyperboliques multidimensionnels. 
Comm. Partial. Differen. Eqns. 14 (1989) 173-230.


\bibitem{Ba}
B.M. Barbashov, V.V. Nesterenko, A.M. Chervyakov, General solutions of nonlinear equations in the geometric theory of the relativistic string. Comm. Math. Phys. 84 (1982) 471-481.









\bibitem{B1}
P. Bizo\'{n}, Equivariant self-similar wave maps from Minkowski spacetime into 3-sphere. Comm. Math. Phys. 215 (2000) 45-56.




\bibitem{BB}
P. Bizo\'{n}, P. Biernt, Generic self-similar blow up for equivariant wave maps and Yang-Mills fields in higher dimensions. Comm. Math. Phys. 338 (2015) 1443-1450.


\bibitem{Born0}
M. Born, L. Infeld, Foundations of the new field theory. Nature. 132 (1933) 1004.





\bibitem{Born}
M. Born, L. Infeld, Foundation of the new field theory. Proc. Roy. Soc. A. 144 (1934) 425-451.

















\bibitem{Chen}
S.X. Chen, Z.G. Feng, The Tricomi problem of a quasi-linear Lavrentiev-Bitsadze mixed type equation.
Z. Angew. Math. Phys. 64 (2013) 755-766.

\bibitem{ChenG}
G.Q. Chen, F.M. Huang, T.Y. Wang, W.Xiang, Steady Euler flows with large vorticity and characteristic discontinuities in arbitrary infinitely long nozzles.
Adv. Math. 346 (2019) 946-1008.


\bibitem{Chr}
D. Christodoulou, Global solutions of nonlinear hyperbolic equations for small initial data. Commun.
Pure Appl. Math. 39 (1986) 267-282.




\bibitem{Egg}
J.Eggers, M.A. Fontelos, The role of self-similarity in singularities of partial
differential equations. Nonlinearity. 22 (2009) R1-R44
















\bibitem{Hop1}
J. Eggers, J. Hoppe, Singularity formation for time-like extremal hypersurfaces. Physics Letters B. 680 (2009) 274-278.


\bibitem{Hop2}
J. Eggers, J. Hoppe, M. Hynek, N. Suramlishvili, Singularities of relativistic membranes. Geom. Flows. 1 (2015) 17-33.











\bibitem{H}
L. H\"{o}rmander,  Implicit function theorems. Stanford Lecture notes, University, Stanford 1977


\bibitem{Han}
Q. Han, Local solutions to a class of Monge-Amp\'{e}re equations of mixed type. Duke. Math. J. 136 (2007) 421-473.

\bibitem{Han1}
Q. Han, M. Khuri, Smooth solutions to a class of mixed type Monge-Amp\'{e}re equations.
Calc. Var. 47 (2013) 825-867.


\bibitem{hop}
J. Hoppe, Some classical solutions of relativistic membrane equations in 4-space-time dimensions. Phys.
Lett. B 329 (1994) 10-14 .


\bibitem{Klai}
S. Klainerman, The null condition and global existence to nonlinear wave equations. Lect. Appl. Math.
23 (1986) 293-326.




 

\bibitem{K}
D.X. Kong, Q. Zhang and Q. Zhou, The Dynamics of Relativistic Strings Moving in the Minkowski Space. Comm. Math. Phys. 269 (2007) 153-174.










\bibitem{Moser}
J. Moser,  A rapidly converging iteration method and nonlinear partial differential equations I-II. \textit{Ann. Scuola Norm. Sup. Pisa.} \textbf{20} (1966) 265-313, 499-535.


\bibitem{Mora}
C. Morawetz, Mixed equations and transonic flow. J. Hyperbolic Differ. Equ. 1 (2004) 1-26.

\bibitem{Nash}
J. Nash, The embedding for Riemannian manifolds. \textit{Amer. Math.} \textbf{63} (1956) 20-63.



\bibitem{N}
T. Nishitani, The cauchy problem for weakly hyperbolic equations of second order. Comm. Partial. Differen. Eqns. 5 (1980) 1273-1296




\bibitem{Ol}
O.A. Oleinik, On the cauchy problem for weakly hyperbolic equations. Commun. Pure Appl. Math. 23 (1970) 569-586.





\bibitem{Lieb}
G.M. Lieberman, Mixed boundary value problems for elliptic and parabolic differential equations of second order.
J. Math Anal. Appl. 113 (1986) 422-440.




\bibitem{Yang2}
F.H. Lin, Y.S. Yang, Gauged harmonic maps, Born-Infeld electromagnetism, and megnetric vortices. Comm. Pure Appl. Math. 56 (2003) 1631-1665.


\bibitem{Lin}
H. Lindblad, A remark on global existence for small initial data of the minimal surface equation in Minkowskian space time. Proc. Amer. Math. Soc. 132 (2004) 1095-1102.






\bibitem{R}
P. Rabinowitz, A rapid convergence method for a singular perturbation problem. Ann. Inst. H. Poincar\'{e} Anal. Non Lin\'{e}aire. 1 (1984) 1-17.


\bibitem{Sog}
C.D. Sogge, 
{\sl Lectures on Nonlinear Wave Equations,} 
Monographs in Analysis, vol. II, International Press, Boston.






\bibitem{tian}
L. Nguyen, G. Tian, On smoothness of timelike maximal cylinders in three-dimensional vacuum spacetimes. Classical Quantum Gravity. 30 (2013), no. 16, 165010, 26 pp.






\bibitem{W}
E. Witten, Singularities in string theory. ICM 2002. Vol I. 495-504.








\bibitem{YAM}
Y. Yamada, Some nonlinear degenerate wave equations. Nonlinear Analysis. TMA. 11 (1987) 1155-1168.












\bibitem{Yan}
W.P. Yan, The motion of closed hypersurfaces in the central force field. J. Diff. Eqns. 261 (2016), 1973-2005.




\bibitem{Yan1}
W.P. Yan, Explicit self-similiar singularity of Born-Infeld equation, space-like surfaces with vanishing mean curvature equation and membrane equation.
arXiv:1712.05159.








\bibitem{Yang}
Y.S. Yang, Classical solutions in the Born-Infeld theory. R. Soc. Lond. Proc. A Math.Phys.
Eng. Sci. 456 (2000) 615-640.































\end{thebibliography}
\end{document}